%% file: rvb-optimalmes-draft.tex
\documentclass[a4paper,11pt]{article}
\usepackage{mathptmx}      
\usepackage{amssymb,amsthm}
\usepackage{graphicx}
\usepackage{fancyhdr}
\usepackage[left=3.3cm,top=3.3cm,bottom=3.3cm,right=3.3cm]{geometry}
\usepackage{hyperref}

\newcommand{\bR}{\mathbb{R}}
\newcommand{\bC}{\mathbb{C}}

\newcommand{\bN}{\mathbb{N}}

\newcommand{\bE}{\mathbb{E}}
\newcommand{\cP}{\mathcal{P}}

\newcommand{\ext}{\mathrm{ext}\,}
\newcommand{\dom}{\mathrm{dom}\,}

\newcommand{\cl}{\mathrm{cl}\,}
\newcommand{\co}{\mathrm{co}\,}

\newcommand{\tr}{\mathrm{tr}\,}
\newcommand{\Int}[1]{\mathrm{Int}({#1})}

\newcommand{\grant}{EPSRC grant EP/H031936/1}

\newtheorem{theorem}{Theorem}
\newtheorem{corollary}{Corollary}

\newtheorem{proposition}{Proposition}
\theoremstyle{definition}
\newtheorem{definition}{Definition}
\newtheorem{example}{Example}

\theoremstyle{remark}
\newtheorem{remark}{Remark}

\begin{document}


\title{Optimal measures and Markov transition kernels\thanks{This work was supported by \grant.}}

\author{Roman V. Belavkin}
%

\date{Draft of 29 November 2011}

\maketitle              

\thispagestyle{fancy}
\lhead{}
\chead{The final publication is available at {\tt www.springerlink.com}\\
In {\em J Glob Optim} DOI: 10.1007/s10898-012-9851-1}
\rhead{} \lfoot{} \cfoot{} \rfoot{}
\renewcommand{\headrulewidth}{.4pt}

\begin{abstract}
We study optimal solutions to an abstract optimization problem for measures, which is a generalization of classical variational problems in information theory and statistical physics.  In the classical problems, information and relative entropy are defined using the Kullback-Leibler divergence, and for this reason optimal measures belong to a one-parameter exponential family.  Measures within such a family have the property of mutual absolute continuity.  Here we show that this property characterizes other families of optimal positive measures if a functional representing information has a strictly convex dual.  Mutual absolute continuity of optimal probability measures allows us to strictly separate deterministic and non-deterministic Markov transition kernels, which play an important role in theories of decisions, estimation, control, communication and computation.  We show that deterministic transitions are strictly sub-optimal, unless information resource with a strictly convex dual is unconstrained.  For illustration, we construct an example where, unlike non-deterministic, any deterministic kernel either has negatively infinite expected utility (unbounded expected error) or communicates infinite information.

\end{abstract}

\section{Introduction}
\label{sec:intro}

This work was motivated by the fact that probability measures within an exponential family, which are solutions to variational problems of information theory and statistical physics, are mutually absolutely continuous.  Thus, we begin by clarifying and discussing this property in the simplest setting.  Let $\Omega$ be a finite set, and let $x:\Omega\rightarrow\bR$ be a real function.  Consider the family $\{y_\beta\}_x$ of real functions $y_\beta:\Omega\rightarrow\bR$, indexed by $\beta\geq0$:
\begin{equation}
y_\beta(\omega)=e^{\beta x(\omega)}\,y_0(\omega)\,,\quad y_0(\omega)\geq0
\label{eq:exponential}
\end{equation}
The elements of $\{y_\beta\}_x$ represent one-parameter exponential measures $y_\beta(E)=\sum_{\omega\in E}y_\beta(\omega)$ on $\Omega$, and normalized elements $P_\beta(\omega)=y_\beta(\omega)/y_\beta(\Omega)$ are the corresponding exponential probability measures.  Of course, exponential measures can be defined on an infinite set, for example, as elements of the Banach space $Y:=\mathcal{M}(\Omega,\bR,\|\cdot\|_1)$ of real Radon measures on a locally compact space $\Omega$ \cite{Bourbaki63}.  In this case, $x$ and $e^x$ are elements of the normed algebra $X:=C_c(\Omega,\bR,\|\cdot\|_\infty)$ of continuous functions with compact support in $\Omega$.  As will be clarified later, $Y$ can be considered not only as the dual of $X$, but also as a module over algebra $X$, which explains the definition of an exponential family~(\ref{eq:exponential}) as multiplication of $y_0\in Y$ by elements of $X$.  Furthermore, for some $y_0$, exponential measures are finite even if function $x$ is not continuous, has non-compact support and unbounded.  A similar construction can be made in the case when $X$ is a non-commutative $\ast$-algebra, such as the algebra of compact Hermitian operators on a separable Hilbert space used in quantum probability theory.  However, quantum exponential measures can be defined in different ways, such as $y_\beta:=\exp(\beta x+\ln y_0)$ or $y_\beta:=y_0^{1/2}\,\exp(\beta x)\,y_0^{1/2}$, which are not equivalent.

One property that characterizes all these exponential measures is that elements within a family are mutually absolutely continuous.  We remind that measure $y$ is absolutely continuous with respect to measure $z$, if $z(E)=0$ implies $y(E)=0$ for all $E$ in the $\sigma$-ring of subsets of $\Omega$.  Mutual absolute continuity is the case when the implication holds in both directions.  It is easy to see from equation~(\ref{eq:exponential}) that exponential measures within one family have exactly the same support and are mutually absolutely continuous.  This property is particularly important, when measures are considered on a composite system, such as a direct product of two sets $\Omega=A\times B$.  Normalized measures on such $\Omega$ are joint probability measures $P(A\times B)$ uniquely defining conditional probabilities $P(A\mid B)$ (i.e. Markov transition kernels).  Observe now that if $P(A\times B)$ and $P(A)P(B)$ (product of marginals) are mutually absolutely continuous, then $P(a\mid b)>0$ for all $a\in A$ such that $P(a)>0$.  Conditional probability with this property is non-deterministic, because several elements $a\in A$ can be in the `image' of $b\in B$.  Clearly, all joint probability measures within an exponential family define such non-deterministic transition kernels.

Another, perhaps the most important, property of exponential families is that they are, in a certain sense, optimal.  It is well-known in mathematical statistics that the lower bound for the variance of the unbiased estimator of an unknown parameter, defined by the Rao-Cramer inequality, is attained if and only if the probability distribution is a member of an exponential family \cite{Cramer46,Rao45}.  In statistical physics, it is known that exponential distributions (i.e. Boltzmann or Gibbs distributions) maximize entropy of a thermodynamical system under a constraint on energy \cite{Jaynes57}.  In information theory, exponential transition kernels are known to maximize a channel capacity \cite{Shannon48,Stratonovich65,Stratonovich75:_inf}, and they are used in some randomized optimization techniques (e.g. \cite{Kirkpatrick83}) as well as various machine learning algorithms \cite{Wainwright-Jordan03}.  A one-parameter exponential family has been studied in information geometry, and it was shown to be a Banach space with an Orlicz norm \cite{Pistone-Sempi95}.  Similar constructions have been considered in quantum probability \cite{Bobkov-Zegalinski05,Streater04}.

Optimality of exponential families of measures on one hand and their mutual absolute continuity on the other is a particularly interesting combination, because it seems that for the first time we have an optimality criterion, with respect to which all deterministic transitions between elements of a composite system are strictly sub-optimal.  This appears to have importance not only for information and communication theories, but also for theories of computational and algorithmic complexity, because Markov transition kernels can be used to represent various input-output systems, including computational systems and algorithms.  Thus, understanding the relation between mutual absolute continuity within some families of measures and their optimality was the main motivation for this work.

It is well-known, and will be reminded later in this paper, that a one-parameter exponential family of probability measures is the solution to a variational problem of minimizing Kullback-Leibler (KL) divergence \cite{Kullback59} of one probability measure from another subject to a constraint on the expected value.  In fact, the logarithmic function, which appears in the definition of the KL-divergence, is precisely the reason why the exponential function appears in the solutions.  However, mutual absolute continuity, which for composite systems implies the non-deterministic property of conditional probabilities, is not exclusive to families of exponential measures.  Indeed, geometrically, this property simply means that measures are in the interior of the same positive cone, defined by their common support.  Thus, our method is based on a generalization of the above mentioned variational problem by relaxing the definition of information and then employing geometric analysis of its solutions.

In the next section, we introduce the notation, define the generalized optimization problem and recall some basic relevant facts.  An abstract information resource will be represented by a closed functional $F:Y\rightarrow\bR\cup\{\infty\}$, defined on the space $Y$ of measures, and such that its values $F(y)$ can be associated with values $I(y,y_0)$ of some information distance (e.g. the KL-divergence).  In Section~\ref{sec:optimality} we establish several properties of optimal solutions.  In particular, we prove in Proposition~\ref{lm:monotone} that the optimal value function is order isomorphism putting information in duality with expected utility of an optimal system.  These results are then used in Section~\ref{sec:main} to prove a theorem relating mutual absolute continuity of optimal positive measures to strict convexity of functional $F^\ast$, the Legendre-Fenchel dual of $F$ representing information resource.  We show that strict convexity of $F^\ast$ is necessary to separate different variational problems by optimal measures, and for this reason it appears to be a natural minimal requirement on information, generalizing the additivity axiom.  Because proof of mutual absolute continuity does not depend on commutativity of algebra $X$, pre-dual of $Y$, these results apply to a general, non-commutative setting used in quantum probability and information theories.  In Section~\ref{sec:kernels}, we discuss optimal Markov transition kernels (conditional probabilities) in the classical (commutative) setting, which is done for simplicity reasons.  We shall recall several facts about transition kernels, information capacity of memoryless channels they represent and the corresponding variational problems.  The main result of this section is a theorem separating deterministic and non-deterministic kernels.  We show how mutual absolute continuity of optimal Markov transition kernels implies that optimal transitions are non-deterministic; deterministic transitions are strictly suboptimal if information, understood broadly here, is constrained.  This result will be illustrated by an example, where any deterministic kernel either has a negatively infinite expected utility (unbounded expected error) or communicates infinite information; a non-deterministic kernel, on the other hand, can have both finite expected utility and finite information.   In the end of the section we shall consider applications of this work to theories of algorithms and computational complexity.  We shall discuss how deterministic and non-deterministic algorithms can be represented by Markov transition kernels between the space of inputs and the space of output sequences, and how constraints on the expected utility or complexity of the algorithms are related to variational problems studied in this work.  The paper concludes by a summary and discussion of the results.

\section{Preliminaries}
\label{sec:preliminaries}

This work is based on a generalization of classical variational problems of information theory and statistical physics, which can be formulated as follows.  Let $(\Omega,\mathcal{R})$ be a measurable set and let $\mathcal{P}(\Omega)$ be the set of all Radon probability measures on $\Omega$.  We denote by $\bE_p\{x\}$ the expected value of random variable $x:\Omega\rightarrow\bR$ with respect to $p\in\mathcal{P}(\Omega)$.  An information distance is a function $I:\mathcal{P}\times\mathcal{P}\rightarrow\bR\cup\{\infty\}$ that is closed (lower semicontinuous) in each argument.  An important example is the Kullback-Leibler divergence $I_{KL}(p,q):=\bE_p\{\ln(p/q)\}$ \cite{Kullback59}.  We remind that $\bE_p\{x\}$ is linear in $p$, and $I_{KL}(p,q)$ is convex.  The variational problem is formulated as follows:
\begin{equation}
\mbox{maximize (minimize)}\quad\bE_p\{x\}\quad\mbox{subject to}\quad\bE_p\{\ln(p/q)\}\leq\lambda
\label{eq:max-i}
\end{equation}
where optimization is over probability measures $p\in\cP$.  This problem can be considered as linear programming with an infinite number of linear constraints, and it can be formulated as the following convex programming problem:
\begin{equation}
\mbox{minimize}\quad\bE_p\{\ln(p/q)\}\quad\mbox{subject to}\quad\bE_p\{x\}\geq\upsilon\quad\Bigl(\bE_p\{x\}\leq\upsilon\Bigr)
\label{eq:min-i}
\end{equation}
Figure~\ref{fig:simplex-kl} illustrates these variational problems on a $2$-simplex of probability measures over a set of three elements with the uniform distribution $q(\omega)=1/3$ as the reference measure.

\begin{figure}[!ht]
\begin{center}
\input{simplex-kl}
\end{center}
\caption{$2$-Simplex $\cP$ of probability measures over set $\Omega=\{\omega_1,\omega_2,\omega_3\}$ with level sets of expected utility $\bE_p\{x\}=\upsilon$ and the Kullback-Leibler divergence $\bE_p\{\ln(p/q)\}=\lambda$.  Probability measure $p_\beta$ is the solution to variational problems~(\ref{eq:max-i}) and (\ref{eq:min-i}).  The family $\{p_\beta\}_x$ of solutions, shown by dashed curve, belongs to the interior of $\cP$.}
\label{fig:simplex-kl}
\end{figure}
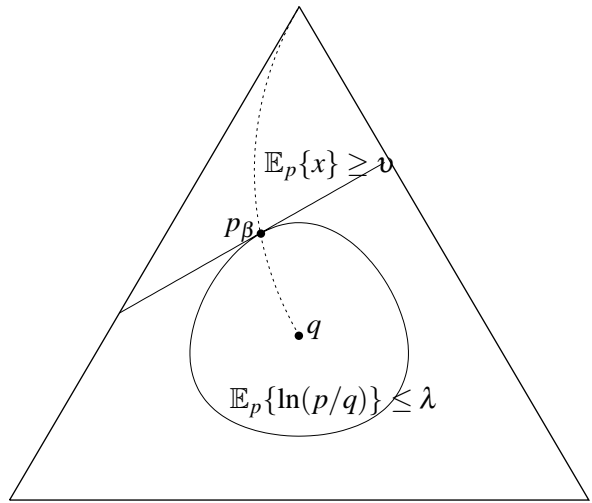

In optimization and information theories, $\bE_p\{x\}$ represents expected utility to be maximized or expected cost to be minimized.  In physics, it represents internal energy.  Information distance $I_{KL}(p,q)$ is also called relative entropy, and the inequality $I_{KL}(p,q)\leq\lambda$ represents an {\em information constraint}.  Depending on the domain of definition of the probability measures, the information constraint may have different meanings, such as a lower bound on entropy (i.e. irreducible uncertainty), partial observability of a random variable, a constraint on the amount of statistical information (i.e. a number of independent tests, questions or bits of information), on communication capacity of a channel, on memory of a computational device and so on \cite{Stratonovich75:_inf}.  These variational problems can also be formulated in quantum physics, where $x$ is an element of a non-commutative algebra of observables, and $p$, $q$ are quantum probabilities (states).

As is well-known, solutions to problems~(\ref{eq:max-i}) and (\ref{eq:min-i}) are elements of an exponential family of probability distributions.  Before we define an appropriate generalization of these problems, we remind some axiomatic principles underpinning the choice of functionals.

\subsection{Axioms behind the choice of functionals}
\label{sec:axioms}

The choice of linear objective functional $\bE_p\{x\}$ has axiomatic foundation in game theory \cite{Neumann-Morgenstern}, where $\Omega$ is equipped with total pre-order $\lesssim$, called the {\em preference relation}, and function $x:\Omega\rightarrow\bR$ is its {\em utility representation}: $\omega_1\lesssim\omega_2$ if and only if $x(\omega_1)\leq x(\omega_2)$.  Because the quotient set $\Omega/\sim$ of a pre-ordered set with a utility function is isomorphic to a subset of the real line, it is separable and metrizable by $\rho([a],[b])=|x(a)-x(b)|$, and therefore every probability measure on the completion of $\Omega/\sim$ is Radon (e.g. by Ulam's theorem for probability measures on Polish spaces).

The set $\cP(\Omega)$ of all classical probability measures on $\Omega$ is a simplex with Dirac measures $\delta_\omega$ comprising the set $\ext\cP$ of its extreme points \cite{Phelps66}.   The question that has been discussed extensively is: How to extend pre-order $\lesssim$, which was defined on $\Omega\equiv\ext\cP$, to the whole $\cP$?  It was shown in \cite{Neumann-Morgenstern} that linear (or affine) functional $\bE_p\{x\}$ is the only functional that makes the extended pre-order $(\cP,\lesssim)$ compatible with the vector space structure of $Y\supset\cP$ and Archimedian.  We remind that for the corresponding pre-order $(Y,\lesssim)\supset(\cP,\lesssim)$ this is defined by the axioms:
\begin{enumerate}
\item $q\lesssim p$ implies $q+r\lesssim p+r$ and $\alpha q\lesssim\alpha p$ for all $r\in Y$ and $\alpha\geq0$.
\item $n q\lesssim p$ for all $n\in\bN$ implies $q\lesssim0$.
\end{enumerate}
In this paper we shall follow this formalism assuming that the objective functional is linear.  We note that non-linearity may arise in certain dynamical systems, where $x$ may change with time, but this will not be considered in this work, because our focus is on optimization problems with respect to some fixed preference relation $\lesssim$ or utility $x$ on $\Omega$.  A non-commutative (quantum) analogue of a utility function was given in \cite{Belavkin09:_qbic09} by a Hermitian operator $x$ on a separable Hilbert space (an observable) with its real spectrum representing a total pre-order on its eigen states.  The principal difference with the classical theory is the existence of incompatible (non-commutative) utility operators.

As mentioned earlier, information constraints may be related to different phenomena (e.g. uncertainty, observability, statistical data, communication capacity, memory, etc).  However, in information theory they often have been represented by functionals, such as relative entropy or Shannon information, which are defined using the Kullback-Leibler divergence $I_{KL}$.  Its choice is also based on a number of axioms \cite{Csiszar91,Khinchin57,Shannon48}, such as additivity: $I_{KL}(p_1p_2,q_1q_2)=I_{KL}(p_1,q_1)+I_{KL}(p_2,q_2)$.  In fact, this axiom is precisely the reason why the logarithm function appears in its definition (i.e. as homomorphism between multiplicative and additive groups of $\bR$).  There is, however, an abundance of other information distances and metrics, such as the Hellinger distance, total variation and the Fisher metrics.  Although they often fail to have a proper statistical interpretation \cite{Chentsov72}, there has been a renewed interest in using different information distances and contrast functions in applications to compare distributions (e.g. see \cite{Amari-Ohara11,Banerjee05,Naudts08}).

For reasons outlined above, we shall generalize problems~(\ref{eq:max-i}) and (\ref{eq:min-i}) by considering an abstract information distance or resource, which will be used to define a subset of feasible solutions.  In addition, we shall not restrict the problems to normalized measures, which makes the exposition a lot simpler.  Normalization can be performed at a later stage.  We now define an appropriate algebraic structure.

\subsection{Dual algebraic structures}
\label{sec:dual-algebras}

Let $X$ and $Y$ be complex linear spaces put in duality via bilinear form $\langle\cdot,\cdot\rangle:X\times Y\rightarrow\bC$:
\[
\langle x,y\rangle=0\,,\ \forall\, x\in X\ \Rightarrow y=0\,,\qquad
\langle x,y\rangle=0\,,\ \forall\, y\in Y\ \Rightarrow x=0
\]
We denote by $X^\sharp$ the algebraic dual of $X$, by $X'$ the continuous dual of a locally convex space $X$ and by $X^\ast$ the complete normed dual space of $(X,\|\cdot\|)$.  The same notation applies to dual spaces of $Y$.  The results will be derived using only the facts that $X$ and $Y$ are ordered linear spaces in duality.  These spaces, however, can have richer algebraic structures, which we briefly outline here.

Space $X$ is closed under an associative, but generally
non-commutative binary operation $\cdot:X\times X\rightarrow X$
(e.g. pointwise multiplication or matrix multiplication) and
involution as a self-inverse, antilinear map $\ast:X\rightarrow X$
reversing the multiplication order: $(x^\ast z)^\ast=z^\ast x$.  Thus,
$X$ is a $\ast$-algebra.  The set of all Hermitian elements $x=x^\ast$ is a real subspace of $X$, and if every $x^\ast x$ has positive real spectrum, then $X$ is called a {\em total} $\ast$-algebra, in which the spectrum of all Hermitian elements is real.  In this case, Hermitian elements $x^\ast x$ form a pointed convex cone $X_+$, generating $X=X_+-X_+$.

The dual space $Y$ is closed under the transposed involution
$\ast:Y\rightarrow Y$, defined by $\langle x,y^\ast\rangle=\langle
x^\ast, y\rangle^\ast$.  It is ordered by a positive cone
$Y_+:=\{y:\langle x^\ast x,y\rangle\geq0\,,\ \forall\,x\in X\}$, dual
of $X_+$, and it has order unit $y_0\in Y_+$ (also called a reference
measure), which is a strictly positive linear functional: $\langle
x^\ast x,y_0\rangle>0$ for all $x\neq0$.  If the pairing
$\langle\cdot,\cdot\rangle$ has the property that for each $z\in X$
there exists a transposed element $z'\in Y$ such that $\langle
zx,y\rangle=\langle x,z' y\rangle$, then $Y\supset X$ is a left
(right) module over $X$ with respect to the transposed left (right)
action $y\mapsto z' y$ ($y\mapsto yz^{\ast\prime\ast}$) of $X$ on $Y$
such that $(xz)'=z'x'$ and $\langle x,yz^{\ast\prime\ast}\rangle=
\langle x^\ast,z^{\ast\prime}y^\ast\rangle^\ast= \langle z^\ast
x^\ast,y^\ast\rangle^\ast=\langle xz,y\rangle$ (see
\cite{vpb11:_qbic10}, Appendix).  In many practical cases, the pairing
$\langle\cdot,\cdot\rangle$ is {\em central} (or {\em tracial}), so
that the left and right transpositions act identically on $y_0$:
$z^{\ast\prime}y_0=y_0z^{\prime\ast}$ for all $z\in X$.  In this case,
the element $z^{\ast\prime}y_0=y_0z^{\prime\ast}\in Y$ can be
identified with a complex conjugation of $z\in X$.

Two primary examples of a total $\ast$-algebra $X$, which are important in this work, are the commutative algebra $C_c(\Omega,\bC,\|\cdot\|_\infty)$ of continuous functions with compact support in a locally compact topological space $\Omega$ and the non-commutative algebra $C_c(\mathcal{H},\bC,\|\cdot\|_\infty)$ of compact Hermitian operators on a separable Hilbert space $\mathcal{H}$.  The corresponding examples of dual space $Y=X^\ast$ are the Banach space $\mathcal{M}(\Omega,\bC,\|\cdot\|_1)$ of complex signed Radon measures on $\Omega$ and its non-commutative generalization $\mathcal{M}(\mathcal{H},\bC,\|\cdot\|_1)$.  Note that these examples of algebra $X$ are generally incomplete and contain only an approximate identity.  However, by $X$ we shall understand here an extended algebra that contains additional elements.  In particular, $X$ will contain the unit element $1\in X$ such that $\langle 1,y\rangle=\|y\|_1$ if $y\geq0$ (i.e. $1\in X$ coincides on $Y_+$ with the norm $\|\cdot\|_1$, which is additive on $Y_+$).  Furthermore, because constraints in variational problems~(\ref{eq:max-i}) or (\ref{eq:min-i}), or their generalizations, define a proper subset of space $Y$, we can consider random variables represented by elements $x\in Y^\sharp$ that are outside of the Banach space $Y^\ast$ (e.g. unbounded functions or operators).

Below are three main examples of pairing $X$ and $Y$ by a sum, an integral or trace:
\begin{equation}
\langle x,y\rangle:=\sum_\Omega x(\omega)\,y(\omega)\,,
\qquad
\langle x,y\rangle:=\int_\Omega x(\omega)\,dy(\omega)\,,
\qquad
\langle x,y\rangle:=\tr\{xy\}
\label{eq:pairing}
\end{equation}
Although the linear functionals $x(y)=\langle x,y\rangle$ are generally complex-valued, we shall assume, without further mentioning, that $\langle\cdot,\cdot\rangle$ is evaluated on Hermitian elements $x=x^\ast$ and $y=y^\ast$ so that $\langle x,y\rangle\in\bR$.  In particular, the expected value $\bE_p\{x\}=\langle x,p\rangle\in\bR$, where $x$ is Hermitian and $p$ is positive.  Thus, the expressions `maximize (minimize) $x(y)=\langle x,y\rangle$' should be understood accordingly as maximization or minimization of a real functional.

\subsection{Generalized variational problems for measures}

Normalized non-negative measures (i.e. probability measures) are elements of the set:
\[
\mathcal{P}:=\{y\in Y:\,y\geq0\,,\ \langle 1,y\rangle=1\}
\]
This is a weakly compact convex set, and therefore $\cP=\cl\co\ext\cP$ by the Krein-Milman theorem.  In the commutative case, $\cP$ is a simplex, because each $p\in\cP$ is uniquely represented by extreme points $\delta\in\ext\cP$ \cite{Phelps66}.  In information geometry $\cP$ is referred to as {\em statistical manifold}, and its topological properties have been studied by defining different information distances $I:\cP\times\cP\rightarrow\bR_+\cup\{\infty\}$ \cite{Amari85,Chentsov72,Pistone-Sempi95}.  We can generalize this by considering information resource as a functional, defined for all positive or Hermitian elements.

\begin{figure}[!ht]
\begin{center}
\input{ui-branches}
\end{center}
\caption{Optimal value functions $\upsilon=\overline x(\lambda)$ and $\upsilon=\underline x(\lambda)$.  The value $\lambda_0=\inf F$ corresponds to $\upsilon\in[\underline\upsilon_0,\overline\upsilon_0]$.  Special values $\overline\lambda$, $\underline\lambda$ of the constraint $\lambda\geq F(y)$ correspond respectively to optimal values $\overline\upsilon$ and $\underline\upsilon$.}
\label{fig:ui-branches}
\end{figure}
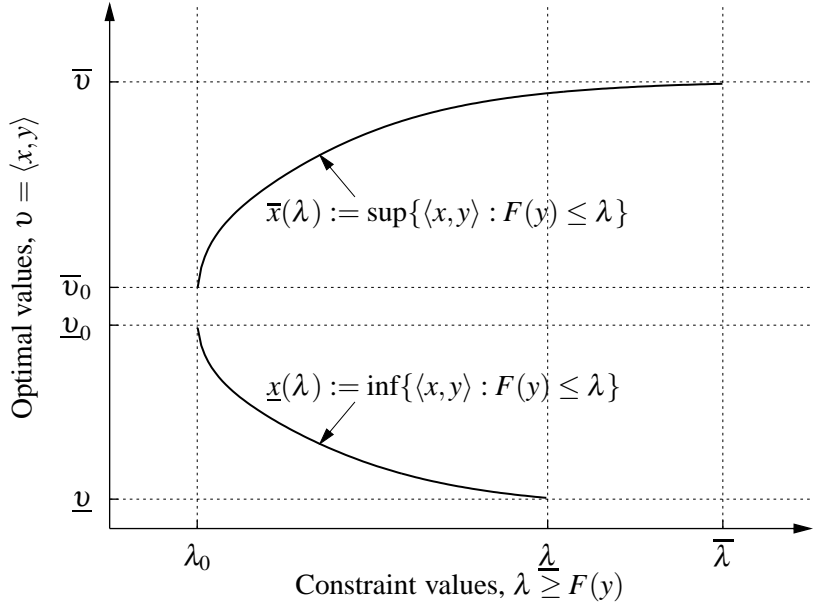

Let $F:Y\rightarrow\bR\cup\{\infty\}$ be a closed functional, so that $F$ is finite at some $y\in Y$, and sublevel sets $\{y:F(y)\leq\lambda\}$ are closed in the weak topology $\sigma(Y,X)$ for each $\lambda$.  Because $-\infty$ is not included in the definition of closed $F$, it is also lower-semicontinuous \cite{Rockafellar74}.  We shall assume without further mentioning that the effective domain $\dom F:=\{y:F(y)<\infty\}$ has non-empty algebraic interior.  In addition, if $Y$ is defined over the field of complex numbers, we shall also assume that $\dom F$ contains only Hermitian elements $y=y^\ast$ (e.g. $\dom F\subseteq Y_+$).

Variational problems~(\ref{eq:max-i}) and (\ref{eq:min-i}) are generalized by considering all, not necessarily positive or normalized measures, and by using any closed functional $F$ to define an information resource.   The optimal values achieved by solutions to these problems are defined by the following {\em optimal value functions}:
\begin{eqnarray}
\overline x(\lambda)&:=&\sup\{\langle x,y\rangle:F(y)\leq\lambda\}
\label{eq:support}\\
\underline x(\lambda)&:=&\inf\{\langle x,y\rangle:F(y)\leq\lambda\}
\label{eq:anti-support}\\
\overline x^{-1}(\upsilon)&:=&\inf\{F(y):\langle x,y\rangle\geq\upsilon\}
\label{eq:inverse-support}\\
\underline x^{-1}(\upsilon)&:=&\inf\{F(y):\langle x,y\rangle\leq\upsilon\}
\label{eq:inverse-anti-support}
\end{eqnarray}
We define $\overline x(\lambda):=-\infty$, if $\lambda<\inf F$, and $\overline x(\infty):=\lim\overline x(\lambda)$ as $\lambda\to\infty$.  Observe that $\underline x(\lambda)=-\overline{(-x)}(\lambda)$ and $\underline x^{-1}(\upsilon)=\overline{(-x)}^{-1}(-\upsilon)$.  Thus, it is sufficient to study only the properties of $\overline x(\lambda)$.  Figure~\ref{fig:ui-branches} depicts schematically the optimal value functions $\overline x(\lambda)$ and $\underline x(\lambda)$.  It is clear from the definition that $\overline x(\lambda)$ is a non-decreasing extended real function, and $\underline x(\lambda)$ is non-increasing.  It will be shown also in the next section that $\overline x(\lambda)$ is concave, and $\underline x(\lambda)$ is convex (Proposition~\ref{lm:monotone}).  Because sets $\{y:F(y)\leq\lambda\}$ may be unbalanced and unbounded, the functions may not be reflections of each other in the sense that $\overline x(\lambda)-\upsilon_0\neq\upsilon_0-\underline x(\lambda)$ for all $\upsilon_0$, and one or both functions can be empty.  The definition of the optimal value functions~(\ref{eq:support})--(\ref{eq:inverse-anti-support}) in terms of functional $F(y)$ of one variable, unlike information distance $I(y,y_0)$, allows for considering the case when $\inf F$ is not achieved at any $y_0\in Y$.

In addition to $\lambda_0:=\inf F$, we define two special values $\overline\lambda$ and $\underline\lambda$ of functional $F$ as follows:
\begin{equation}
\overline x(\overline\lambda):=\sup\{\langle x,y\rangle:y\in\dom F\}\,,\qquad
\underline x(\underline\lambda):=\inf\{\langle x,y\rangle:y\in\dom F\}
\label{eq:lambda_x}
\end{equation}
Thus, problems of maximization or minimization of $x(y)=\langle x,y\rangle$ subject to constraints $F(y)\leq\overline\lambda$ or $F(y)\leq\underline\lambda$ respectively are equivalent to unconstrained problems on $\dom F$.  The corresponding optimal values are denoted $\overline\upsilon=\overline x(\overline\lambda)$ and $\underline\upsilon=\underline x(\underline\lambda)$, as shown on Figure~\ref{fig:ui-branches}. The reason for defining these values is that generally $\overline\lambda\leq\infty$, $\underline\lambda\leq\infty$ and $\overline\lambda\neq\underline\lambda$ (see Figure~\ref{fig:ui-branches}).  Solutions to unconstrained problems may correspond to large, possibly infinite values $\overline\lambda$ or $\underline\lambda$, and therefore they can be considered unfeasible.  Subsets of feasible solutions will be defined by constraints $F(y)\leq\lambda<\overline\lambda$ or $F(y)\leq\lambda<\underline\lambda$.

In addition, we define the following special values:
\begin{equation}
\overline\upsilon_0:=
\lim_{\lambda\downarrow\inf F}\sup\{\langle x,y\rangle:F(y)\leq\lambda\}\,,\qquad
\underline\upsilon_0:=
\lim_{\lambda\downarrow\inf F}\inf\{\langle x,y\rangle:F(y)\leq\lambda\}
\label{eq:upsilon_0}
\end{equation}
If there exists a set $\partial F^\ast(0)\subset\dom F$ such that $\inf F=F(y_0)$ for all $y_0\in\partial F^\ast(0)$, then $\overline\upsilon_0=\sup\{\langle x,y_0\rangle:y_0\in\partial F^\ast(0)\}$ and $\underline\upsilon_0=\inf\{\langle x,y_0\rangle:y_0\in\partial F^\ast(0)\}$.  If $y_0$ is unique, then $\overline\upsilon_0=\underline\upsilon_0$; otherwise $\overline\upsilon_0\geq\underline\upsilon_0$ (see Figure~\ref{fig:ui-branches}).  Elements $y_0\in\partial F^\ast(0)$ represent trivial solutions, because they correspond to constraint $\lambda_0:=\inf F$ in functions $\overline x(\lambda)$ and $\underline x(\lambda)$.  Constraints $\langle x,y\rangle\geq\upsilon>\overline\upsilon_0$ and $\langle x,y\rangle\leq\upsilon<\underline\upsilon_0$ in the inverse functions $\overline x^{-1}(\upsilon)$ and $\underline x^{-1}(\upsilon)$ ensure that $F(y)>\lambda_0$, and the solutions are non-trivial.

\subsection{Some facts about subdifferentials of dual convex functions}

In the next section, we show that solutions to the generalized variational problems with optimal values~(\ref{eq:support})--(\ref{eq:inverse-anti-support}), if exist, are elements of a subdifferential of functional $F^\ast$, dual of $F$.  We remind that $F^\ast:X\rightarrow\bR\cup\{\infty\}$ is the Legendre-Fenchel transform of $F$:
\[
F^\ast(x):=\sup\{\langle x,y\rangle-F(y)\}
\]
and it is aways closed and convex (e.g. see
\cite{Rockafellar74,Tikhomirov90:_convex}).  Condition
$F^{\ast\ast}=F$ implies $F$ is closed and convex.  Otherwise, the
epigraph of $F^{\ast\ast}$ is a convex closure of the epigraph of $F$
in $Y\times\bR$.  Closed and convex functionals are continuous on the
(algebraic) interior of their effective domains (e.g. see
\cite{Moreau67} or \cite{Rockafellar74}, Theorem~8), and they have
the property
\begin{equation}
x\in\partial F(y)\quad\iff\quad\partial F^\ast(x)\ni y
\label{eq:subdifferentials}
\end{equation}
where set $\partial F(y):=\{x:\langle x,z-y\rangle\leq F(z)-F(y)\,,\ \forall\,z\in Y\}$ is {\em subdifferential} of $F$ at
$y$, and its elements are called {\em subgradients}.  In particular,
$0\in \partial F(y_0)$ implies $F(y_0)\leq F(y)$ for all $y$
(i.e. $\inf F=F(y_0)$).  We point out that the notions of subgradient
and subdifferential make sense even if $F$ is not convex or finite at
$y$, but non-empty $\partial F(y)$ implies $F(y)<\infty$ and
$F(y)=F^{\ast\ast}(y)$, $\partial F(y)=\partial F^{\ast\ast}(y)$ (\cite{Rockafellar74}, Theorem~12).\footnote{It is possible, however, that $F(y)<\infty$, but $\partial F(y)=\varnothing$ (e.g. see \cite{Tikhomirov90:_convex}, Chapter~1, Section~2.4, Example~6d).}  Functional $F^\ast$ is strictly convex if and only if $\partial F^\ast(x)\ni y$ is injective, so that the inverse mapping $\partial F(y)=\{x\}$ is single-valued.

Recall also that subdifferential $\partial F^\ast:X\rightarrow 2^Y$ of a convex function is an example of monotone operator \cite{Kachurovskii68}:
\begin{equation}
\langle x_1-x_2,y_1-y_2 \rangle \geq0\,,\quad\forall\,y_i\in\partial F^\ast(x_i)
\label{eq:monotonicity}
\end{equation}
The inequality is strict for all $x_1\neq x_2$ if and only if $\partial F^\ast(x)\ni y$ is injective (i.e. $\partial F^\ast$ is strictly monotone).

We remind also that $H:Y\rightarrow\bR\cup\{-\infty\}$ is {\em
  concave} if $F(y)=-H(y)$ is convex.  The dual of $H$ in concave
sense is $H^\ast(x):=\inf\{\langle x,y\rangle-H(y)\}$.  By analogy,
one defines {\em supgradient} and {\em supdifferential} of a concave
function \cite{Rockafellar74}.

\section{General properties of optimal solutions and the optimal value functions}
\label{sec:optimality}

In this section, we apply the standard method of Lagrange multipliers
to derive solutions $y_\beta$ achieving the optimal value $\overline
x(\lambda)=\langle x,y_\beta\rangle$.  Then we shall study existence
of solutions and monotonic properties of the optimal value
functions~(\ref{eq:support})--(\ref{eq:inverse-anti-support}).

\subsection{Optimality conditions}

\begin{proposition}[Necessary and sufficient optimality conditions]
Element $y_\beta\in Y$ maximizes linear functional $x(y)=\langle
x,y\rangle$ on sublevel set $\{y:F(y)\leq\lambda\}$ of a closed
functional $F:Y\rightarrow\bR\cup\{\infty\}$ if and only if the
following conditions hold
\[
y_\beta\in\partial F^\ast(\beta x)\,,\qquad F(y_\beta)=\lambda
\]
where parameter $\beta^{-1}>0$ is related to $\lambda$ via $\beta^{-1}\in \partial\overline x(\lambda)$.
\label{lm:optimality-conditions}
\end{proposition}

\begin{proof}
If $y_\beta$ maximizes $\langle x,y\rangle$ on sublevel set $C(\lambda):=\{y:F(y)\leq\lambda\}$, then it belongs to the boundary of $C(\lambda)$ (because $\langle x,\cdot\rangle$ is linear and $C(\lambda)$ is closed).  Moreover, $y_\beta$ belongs also to the boundary of a convex closure of $C(\lambda)$, because it is the intersection of all closed half-spaces $\{y:\langle x,y\rangle\leq\langle x,y_\beta\rangle\}$ containing $C(\lambda)$.  Observe also that
\[
\cl\co\{y:F(y)\leq\lambda\}=\{y:F^{\ast\ast}(y)\leq\lambda\}
\]
and therefore solutions satisfy condition
$F(y_\beta)=F^{\ast\ast}(y_\beta)$ and $\partial F(y_\beta)=\partial
F^{\ast\ast}(y_\beta)$ (e.g. see \cite{Rockafellar74}, Theorem~12).
Thus, the Lagrange function for the conditional extremum
in~(\ref{eq:support}) can be written in terms of $F^{\ast\ast}$ as
follows
\[
K(y,\beta^{-1})=\langle x,y\rangle+\beta^{-1}[\lambda-F^{\ast\ast}(y)]\,,
\]
where $\beta^{-1}$ is the Lagrange multiplier for the constraint
$\lambda\geq F^{\ast\ast}(y)$.  This Lagrange function is concave for
$\beta^{-1}\geq0$, and therefore condition $\partial
K(y_\beta,\beta^{-1})\ni 0$ is both necessary and sufficient for
$y_\beta$ and $\beta^{-1}$ to define its least upper bound, which
gives
\begin{eqnarray*}
\partial_y K(y_\beta,\beta^{-1})=x-\beta^{-1}\partial F^{\ast\ast}(y_\beta)\ni 0\,,
&\quad\Rightarrow\quad&y_\beta\in\partial F^\ast(\beta x)\\
\partial_{\beta^{-1}} K(y_\beta,\beta^{-1})=\lambda-F^{\ast\ast}(y_\beta)=0\,,
&\quad\Rightarrow\quad& F^{\ast\ast}(y_\beta)=\lambda
\end{eqnarray*}
Note that if $F\neq F^{\ast\ast}$, then generally $F^{\ast\ast}(y)\leq
F(y)$, and condition $F^{\ast\ast}(y_\beta)=\lambda$ must be replaced
by a stronger condition $F(y_\beta)=\lambda$.

Noting that $\overline x(\lambda)=\langle
x,y_\beta\rangle+\beta^{-1}[\lambda-F(y_\beta)]$, the Lagrange
multiplier is defined by $\partial\overline x(\lambda)\ni\beta^{-1}$.
Note that $\partial\overline x(\lambda)\geq0$, because $\overline
x(\lambda)$ is non-decreasing, and $\beta^{-1}=0$ if and only if
$F(y)\geq\overline\lambda$.
\end{proof}

\begin{remark}
The inverse optimal value $\overline x^{-1}(\upsilon)$, defined by equation~(\ref{eq:inverse-support}), is achieved by solutions $y_\beta$ given by similar
conditions.  Indeed, the corresponding Lagrange function is
\[
K(y,\beta)=F^{\ast\ast}(y)+\beta[\upsilon-\langle x,y\rangle]
\]
and the necessary and sufficient conditions are
\[
y_\beta\in\partial F^\ast(\beta x)\,,\qquad
\langle x,y_\beta\rangle=\upsilon
\]
where $\beta>0$ is related to $\upsilon$ via $\beta\in
\partial\overline x^{-1}(\upsilon)$.  We note also that conditions for optimal values  $\underline x(\lambda)=-\overline{(-x)}(\lambda)$ and $\underline x^{-1}(\upsilon)=\overline{(-x)}^{-1}(-\upsilon)$, defined by equations~(\ref{eq:anti-support}) and (\ref{eq:inverse-anti-support}), are identical to those in Proposition~\ref{lm:optimality-conditions} and above with the exceptions that $\beta^{-1}<0$ and $\beta<0$.
\label{eq:inverse-optimality-conditions}
\end{remark}

\subsection{Existence of solutions}

The existence of optimal solutions in Proposition~\ref{lm:optimality-conditions} is equivalent to finiteness of $\overline x(\lambda)$, which depends on the properties of sublevel set $C(\lambda):=\{y:F(y)\leq\lambda\}$ and linear functional $x(y)=\langle x,y\rangle$.  Clearly, the existence of solutions is guaranteed if $C(\lambda)$ is bounded in $(Y,\|\cdot\|)$ and $x\in Y^\ast$.  This setting, however, appears to be too restrictive.  First, the restriction of $x$ to Banach space $Y^\ast$ is not desirable in many applications.  Indeed, measures are often considered as elements of a Banach space with norm $\|\cdot\|_1$ of absolute convergence, and therefore $Y^\ast$ is complete with respect to the Chebyshev (supremum) norm $\|\cdot\|_\infty$.  Many objective functions, however, such as utility or cost functions, are expressed using unbounded forms, such as polynomials, logarithms and exponentials.  Second, the sublevel sets $C(\lambda)$ are generally unbalanced (i.e. if $I(y,y_0)\neq I(y_0,y)$ or $F(y_0+[y-y_0])\neq F(y_0-[y-y_0])$), which means that $\overline x(\lambda)\neq\overline{(-x)}(\lambda)$, and therefore $\overline x(\lambda)\in\bR$ does not imply $\overline{(-x)}(\lambda)\in\bR$.  In addition, sets $C(\lambda)$ can be unbounded in $(Y,\|\cdot\|)$ if we allow for measures that are not necessarily normalized.  In this case, finiteness of $\overline x(\lambda)$ is no longer guaranteed, even if $x\in Y^\ast$.  These considerations motivate us to define the most general class of linear functionals $x\in Y^\sharp$ (elements of algebraic dual) that admit optimal solutions to the generalized variational problems for measures and achieving finite optimal values for all constraints.

\begin{definition}[$F$-bounded linear functional]
An element $x\in Y^\sharp$ is bounded above (below) relative to a closed functional $F:Y\rightarrow\bR\cup\{\infty\}$ or {\em $F$-bounded above} ({\em below}) if it is bounded above (below) on sets $\{y:F(y)\leq\lambda\}$ for each $\lambda\in(\lambda_0,\overline\lambda)$ ($\lambda\in(\lambda_0,\underline\lambda)$).  We call $x\in Y^\sharp$ {\em $F$-bounded} if it is $F$-bounded above and below.
\end{definition}

Thus, bounded linear functionals $x\in Y^\ast$ are $\|\cdot\|$-bounded.  If $F(y)=I(y,y_0)$ is understood as information, then we speak of information-bounded functionals.  Although we do not address topological questions in this paper, we point out that the values $\overline x(\lambda)$ coincide with the values of support function $s_{C(\lambda)}(x):=\sup\{\langle x,y\rangle:y\in C(\lambda)\}$ of set $C(\lambda)$, and it generalizes a seminorm on $Y'$.  In fact, a seminorm can be defined for $F$-bounded elements as $\sup\{-\underline x(\lambda),\overline x(\lambda)\}=\sup\{s_{C(\lambda)}(-x),s_{C(\lambda)}(x)\}$, which means they form a topological vector space.  There are, however, elements $x\in Y^\sharp$ that are only $F$-bounded above or below, as will be illustrated in the next example.

\begin{example}
Let $\Omega=\bN$ and let $X$, $Y$ be the spaces of real sequences
$\{x(n)\}$ and $\{y(n)\}$ with pairing $\langle\cdot,\cdot\rangle$
defined by the sum~(\ref{eq:pairing}).  Let $F(y)=\langle\ln
y-1,y\rangle$ for $y>0$, so that the gradient $\nabla F(y)=\ln y$, and
$F$ is minimized at the counting measure $y_0(n)=1$.  The optimal
solutions have the form $y_\beta=e^{\beta x}$, and the values of
functions $\overline x(\lambda)$ and $\underline
x(\lambda)=-\overline{(-x)}(\lambda)$ are respectively
\[
\langle x,y_\beta\rangle=\sum_{n=1}^\infty x(n)\,e^{\beta x(n)}
\quad\mbox{and}\quad
\langle x,y_\beta\rangle=\sum_{n=1}^\infty x(n)\,e^{-\beta x(n)}\,,\qquad
\beta^{-1}>0
\]
In particular, for $x(n)=-n$, the first series converges to
$-e^{\beta}(e^{\beta}-1)^{-2}$, but the second diverges for any
$\beta^{-1}>0$.  Thus, $x(n)=-n$ is $F$-bounded above, but not below.
Observe also that $x(n)=-n$ is unbounded, because
$\|x\|_\infty:=\sup\{|\langle x,y\rangle|:\|y\|_1\leq1\}$ is infinite.
On the other hand, any constant sequence $x(n)=\alpha\in\bR$ is
bounded ($\|x\|_\infty=|\alpha|$), but it is not $F$-bounded above
or below.
\label{ex:f-bounded}
\end{example}

The criterion for element $x\in Y^\sharp$ to be $F$-bounded above follows from the optimality conditions, obtained in Proposition~\ref{lm:optimality-conditions}.

\begin{proposition}[Existence of solutions]
Solutions $y_\beta\in Y$ maximizing $x(y)=\langle x,y\rangle$ on sets
$\{y:F(y)\leq\lambda\}$ exist for all values $\lambda\in(\lambda_0,\overline\lambda)$ of a closed functional $F:Y\rightarrow\bR\cup\{\infty\}$, if there exists at least one number $\beta^{-1}>0$ such that subdifferential $\partial F^\ast(\beta x)$ is non-empty.
\label{lm:existence}
\end{proposition}

\begin{proof}
The element $y_\beta\in\partial F^\ast(\beta x)$ maximizes $x(y)=\langle x,y\rangle$ on $\{y:F(y)\leq\lambda\}$ by Proposition~\ref{lm:optimality-conditions}, and if $\beta^{-1}>0$ and $x\neq0$, then $F(y_\beta)=\lambda\in(\lambda_0,\overline\lambda)$.  The optimal value $\overline x(\lambda)\in\bR$ is equal to
\[
\langle x,y_\beta\rangle=\beta^{-1}\left[F^\ast(\beta x)+F(y_\beta)\right]
\]
Note also that $F^\ast(\beta x)\in (\inf F^\ast,\sup F^\ast)$.  Because sets $\{y:F(y)\leq\lambda\}$ are closed for all $\lambda$ ($F$ is closed), the existence of a solution for one $\lambda$ implies the existence of solutions for all $\lambda$, and they are $y_\beta\in\partial F^\ast(\beta x)$ enumerated by different values $\beta^{-1}>0$.
\end{proof}

Thus, element $x\in Y^\sharp$ is $F$-bounded above if $\partial F^\ast(\beta x)$ is non-empty at least for one $\beta^{-1}>0$.  Geometrically, this means that $x$ can be absorbed into the convex set $C^\ast(\lambda^\ast):=\{w:F^\ast(w)\leq\lambda^\ast\}$ for some $\lambda^\ast\in(\inf F^\ast,\sup F^\ast)$.  If $x\in Y^\sharp$ is also $F$-bounded below, then $-x$ can be absorbed into $C^\ast(\lambda^\ast)$.  Therefore, if $x\in Y^\sharp$ is $F$-bounded only above or below, then the origin of a one-dimensional subspace $\bR x:=\{\beta x:\beta\in\bR\}$ is not on the interior of $\dom F^\ast$.  In fact, it is well-known that if sets $C(\lambda):=\{y:F(y)\leq\lambda\}$ are bounded, then $0\in\Int{\dom F^\ast}$ (see \cite{Asplund-Rockafellar69,Moreau67}).

\subsection{Monotonic properties}

\begin{proposition}[Monotonicity]
Optimal value functions $\overline x(\lambda)$, $\underline x(\lambda)$, $\overline
x^{-1}(\upsilon)$ and $\underline x^{-1}(\upsilon)$, defined by
equations~(\ref{eq:support}), (\ref{eq:anti-support}),
(\ref{eq:inverse-support}) and (\ref{eq:inverse-anti-support}) for a
closed $F:Y\rightarrow\bR\cup\{\infty\}$ and $x\neq0$, have the
following properties:
\begin{enumerate}
\item The mapping $\lambda\mapsto\beta^{-1}\in\partial\overline
  x(\lambda)$ is non-increasing, and $\upsilon\mapsto\beta\in\partial\overline
  x^{-1}(\upsilon)$ is non-decreasing.
\item If in addition $F^\ast$ is strictly convex, then these mappings
  are differentiable so that $\beta^{-1}=d\overline x(\lambda)/d\lambda$ and  $\beta=d\overline x^{-1}(\upsilon)/d\upsilon$.
\item $\overline x(\lambda)$ is concave and strictly increasing for
  $\lambda\in[\lambda_0,\overline\lambda]$.
\item $\underline x(\lambda)$ is convex and strictly decreasing for
  $\lambda\in[\lambda_0,\underline\lambda]$.
\item $\overline x^{-1}(\upsilon)$ is convex and strictly increasing
for $\upsilon\in[\overline\upsilon_0,\overline\upsilon]$.
\item $\underline x^{-1}(\upsilon)$ is convex and strictly decreasing
for $\upsilon\in[\underline\upsilon,\underline\upsilon_0]$.
\end{enumerate}
where $\overline\lambda$, $\underline\lambda$ are defined by equations~(\ref{eq:lambda_x}), and $\overline\upsilon_0$, $\underline\upsilon_0$ by equations~(\ref{eq:upsilon_0}).
\label{lm:monotone}
\end{proposition}

\begin{proof}
\begin{enumerate}
\item Let $y_{\beta_1}$, $y_{\beta_2}$ be maximizers of linear
  functional $x(y)=\langle x,y\rangle$ on sublevel sets $\{y:F(y)\leq\lambda\}$
  with constraints $\lambda_1$, $\lambda_2$ respectively, and let $\upsilon_1=\langle x,y_{\beta_1}\rangle$ and $\upsilon_2=\langle x,y_{\beta_2}\rangle$ denote the corresponding optimal values.
  Clearly, $\lambda_1\leq\lambda_2$ implies $\upsilon_1\leq\upsilon_2$
  by the inclusion
  $\{y:F(y)\leq\lambda_1\}\subseteq\{y:F(y)\leq\lambda_2\}$, so that
  the optimal value function $\overline x(\lambda)=\langle
  x,y_\beta\rangle$ is non-decreasing.  Using condition
  $y_\beta\in\partial F^\ast(\beta x)$ of
  Proposition~\ref{lm:optimality-conditions} and monotonicity
  condition~(\ref{eq:monotonicity}) for convex $F^\ast$, we have
\[
\langle\beta_2 x-\beta_1 x,y_{\beta_2}-y_{\beta_1}\rangle
=(\beta_2-\beta_1)\langle x,y_{\beta_2}-y_{\beta_1}\rangle\geq0
\]
Therefore, $\upsilon_1\leq\upsilon_2$ implies $\beta_1\leq\beta_2$.
This proves that $\lambda\mapsto\beta^{-1}$ is non-increasing, and
$\upsilon\mapsto\beta$ is non-decreasing.

\item Optimality condition $y_\beta\in\partial F^\ast(\beta x)$ is
  equivalent to $\beta x\in\partial F(y_\beta)$ by
  property~(\ref{eq:subdifferentials}), and together with condition
  $F(y_\beta)=\lambda$ or $\langle x,y_\beta\rangle=\upsilon$ it
  implies that different $\beta_1<\beta_2$ can correspond to the same
  $\lambda$ or $\upsilon$ if and only if $\partial F(y_\beta)$
  includes both $\beta_1x$ and $\beta_2x$.  This implies that $F^\ast$
  is not strictly convex on $[\beta_1x,\beta_2x]\subseteq\partial
  F(y_\beta)$.  Dually, if $F^\ast$ is strictly convex, then
  $\beta_1\neq\beta_2$ implies $\lambda_1\neq\lambda_2$ and
  $\upsilon_1\neq\upsilon_2$, so that
  $\{\beta^{-1}\}=\partial\overline x(\lambda)$ and
  $\{\beta\}=\partial\overline x^{-1}(\upsilon)$.  In this case,
  monotone functions $\overline x(\lambda)$ and $\overline
  x^{-1}(\upsilon)$ are differentiable.

\item Function $\overline x(\lambda)$ is strictly increasing on
  $\lambda\in[\lambda_0,\overline\lambda]$, because $\partial\overline
  x(\lambda)\ni\beta^{-1}\geq0$ and $\beta^{-1}=0$ if and only if
  $\lambda\geq \overline\lambda$ (Proposition~\ref{lm:optimality-conditions}).
  The mapping $\lambda\mapsto\beta^{-1}\in\partial\overline
  x(\lambda)$ is non-increasing, and therefore $\overline x(\lambda)$
  is concave.

\item By the same reasoning as above, function
  $\overline{(-x)}(\lambda)$ is concave and strictly increasing for
  $\lambda\in[\lambda_0,\underline\lambda]$.  Thus, $\underline
  x(\lambda)=-\overline{(-x)}(\lambda)$ is convex and strictly
  decreasing.

\item Function $\overline x^{-1}(\upsilon)$ is strictly increasing for all
  $\upsilon\in[\overline\upsilon_0,\overline\upsilon]$, because $\partial\overline
  x^{-1}(\upsilon)\ni\beta\geq0$, and $\beta=0$ if and only if
  $\upsilon=\langle x,y_0\rangle\leq\overline\upsilon_0$ for any
  $y_0\in\partial F^\ast(0)$ ($\lambda_0:=\inf F=F(y_0)$).  Moreover, the mapping
  $\upsilon\mapsto\beta\in\partial\overline x^{-1}(\upsilon)$ is
  non-decreasing, and therefore $\overline x^{-1}(\upsilon)$ is
  convex.

\item Function $\underline x^{-1}(\upsilon)$ is the inverse of convex
  and strictly decreasing function $\underline x(\lambda)$.  Thus,
  $\underline x^{-1}(\upsilon)$ is also convex and strictly decreasing
  for $\upsilon\in[\underline\upsilon,\underline\upsilon_0]$.
\end{enumerate}
\end{proof}

We now use the facts that $X$ is ordered by a pointed convex cone $X_+$, generating $X=X_+-X_+$, and that $Y$ is ordered by the dual cone:
$Y_+:=\{y\in Y:\langle x,y\rangle\geq0\,,\ \forall\,x\geq0\}$.  For example, this is the case when $X$ is a function space with the pointwise order, or if $X$ is the space of operators on a Hilbert space  with $x^\ast x\in X_+$.

\begin{proposition}[Zero solution]
Let $X$ be ordered by a generating pointed cone $X_+$, and let $\{y_\beta\}_x$ be the family of all elements maximizing linear functional $x(y)=\langle x,y\rangle$ on sets $\{y:F(y)\leq\lambda\}$ for all values $\lambda$ of a closed functional $F:Y\rightarrow\bR\cup\{\infty\}$.  If all $y_\beta\in\{y_\beta\}_x$ are non-negative and $y_\beta=0$ for some $\lambda$, then
\[
x=0\quad\mbox{or}\quad F(0)=\lambda_0\quad\mbox{or}\quad F(0)=\overline\lambda
\]
where $\lambda_0:=\inf F$, and $\overline\lambda$ is such that $\overline x(\overline\lambda)=\sup\{\langle x,y\rangle:y\in\dom F\}$
\label{pr:zero}
\end{proposition}

\begin{proof}
Assume the opposite: $x\neq0$ and $\lambda_0<F(0)<\overline\lambda$.  Then
function $\overline x(\lambda)=\langle x,y_\beta\rangle$ is strictly
increasing (Proposition~\ref{lm:monotone}), and sets $\{y:F(y)<F(0)\}$
and $\{y:F(0)<F(y)\}$ are non-empty ($F$ is closed).  Thus, there
exist solutions $y_1$ and $y_2$ such that
\[
F(y_1)<F(0)<F(y_2)\quad\mbox{and}\quad
\langle x,y_1\rangle<0<\langle x,y_2\rangle
\]
Using decomposition $x=x_+-x_-$, $x_+$, $x_-\in X_+$ and $y_1$,
$y_2\in Y_+$, we conclude that
\[
\langle x_+-x_-,y_1\rangle<0<\langle x_+-x_-,y_2\rangle
\quad\Rightarrow\quad
x_+>x_-\quad\mbox{and}\quad x_+<x_-
\]
This implies $x=0$, which is a contradiction.
\end{proof}

\section{Optimal measures}
\label{sec:main}

Our interest is in the support set of optimal positive measures
maximizing linear functional $x(y)=\langle x,y\rangle$ on closed sets
$\{y:F(y)\leq\lambda\}$.  First, we shall prove the main theorem about mutual absolute continuity within families of optimal measures.  Then we shall discuss the underlying property of an information functional.  In the end of this section, we formulate a corollary stating that support of a utility function or operator is contained in the support of optimal measures.

\subsection{Mutual absolute continuity of optimal measures}

Let $X$ be a $\ast$-algebra with a unit element $1\in X$.  Recall that $X$ can be associated with the algebra $\mathcal{R}(\Omega)$ of subsets of $\Omega$ in the classical (commutative) setting, or with the algebra $\mathcal{R}(\mathcal{H})$ of operators on a Hilbert space $\mathcal{H}$ in the non-classical (non-commutative) setting.  A subalgebra $\mathcal{R}(E)$ of subset $E\subset\Omega$ or subspace $E\subset\mathcal{H}$ corresponds in each case to a subalgebra $M\subset X$, and we shall use notation $y(M)=0$ to denote measures that are zero on subset or subspace $E$.  The dual of subalgebra $M\subset X$ is the factor space $Y/M^\bot$ of equivalence classes $[y]:=\{z\in Y:y-z\in M^\bot\}$ generated by the annihilator $M^\bot:=\{y\in Y:\langle x,y\rangle=0\,,\,\forall\,x\in M\}$.  Thus, the elements of $Y/M^\bot$ correspond to measures that are equivalent on $M$, and $M^\bot=[0]\in Y/M^\bot$ is the subspace of measures $y(M)=0$.

We shall define the restriction of functions or operators $x$ to subset or subspace $E$ as their localization $\Pi_M x$, where $\Pi_M:X\rightarrow M$ is a positive `super' operator (i.e. a linear operator acting on the algebra of functions or operators) such that $\Pi_M(X)=M$ and $\Pi_M(x^\ast x)\geq 0$.  Note that when $X$ is a commutative algebra, one can always define $\Pi_M$ with the projection property $\Pi_M^2=\Pi_M$, leaving $M$ invariant.  In the non-commutative case, a projection of $X$ onto $M$ exists if and only if $M$ is invariant under the action of a modular automorphism group (see \cite{Takesaki72} for details).  More specifically, the positive operator $\Pi_M$ satisfies in this case condition $\Pi_M(wx)=w\Pi_M(x)$ for all $w\in M$ and all $x\in X$.  If in addition $\Pi_M(1)=1$, then $\Pi_M$ is the non-commutative generalization of conditional expectation (e.g. see \cite{Petz88:_cond}).  Clearly, only subalgebras $M\subset X$ with projections have statistical or physical meaning.  Note that one can always construct a completely positive linear operator $\Pi_M$, which becomes a projection onto $M$, if $M$ has the above mentioned property of modular automorphism invariance \cite{Accardi-Cecchini82}.  We shall refer to such $\Pi_M$ as {\em localization} onto subalgebra $M$.  The restriction of $F^\ast:X\rightarrow\bR\cup\{\infty\}$ to $M$ is given by $F^\ast(\Pi_Mx)$, and the dual of $F^\ast(\Pi_Mx)$ is defined on $Y/M^\bot$ as $F^{\ast\ast}([y]):=\inf\{F^{\ast\ast}(y):y\in[y]\}$.

\begin{theorem}[Mutual absolute continuity]
Let $X$ be ordered by a generating pointed cone $X_+$, and let $\{y_\beta\}_x$ be the family of all elements maximizing linear functional $x(y)=\langle x,y\rangle$ on sets $\{y:F(y)\leq\lambda\}$ for all values $\lambda$ of a closed functional $F:Y\rightarrow\bR\cup\{\infty\}$.  If all $y_\beta\in\{y_\beta\}_x$ are non-negative and $F^\ast(x):=\sup\{\langle x,y\rangle-F(y)\}$ is strictly convex, then:
\begin{enumerate}
\item There is a subfamily $\{y^\circ_\beta\}_x\subseteq\{y_\beta\}_x$
  containing $y^\circ_\beta$ for each $\lambda\in(\lambda_0,\overline\lambda)$, and $y^\circ_\beta$ correspond to mutually
  absolutely continuous positive measures.
\item If there exists element $y_0$ (resp. $\delta_x$) in
  $\{y_\beta\}_x$ such that $\inf F=F(y_0)$ (resp. $\sup\{\langle
  x,y\rangle:y\in\dom F\}=\langle x,\delta_x\rangle$), then $y_0$
  (resp. $\delta_x$) is absolutely continuous w.r.t. all $y^\circ_\beta$.
\item If in addition $F^{\ast\ast}$ is strictly convex, then
  $\{y^\circ_\beta\}_x=\{y_\beta\}_x\setminus\{y_0,\delta_x\}$.
\end{enumerate}
\label{th:absolute}
\end{theorem}

\begin{proof}
Let $y_\beta$ be a solution for some $\lambda\in(\lambda_0,\overline\lambda)$.  Then $y_\beta\in\partial F^\ast(\beta x)$,
$0<\beta^{-1}<\infty$ (Proposition~\ref{lm:optimality-conditions}).
Let $\Pi_M:X\rightarrow M$ be a localization operator onto subalgebra $M\subset X$ (i.e. a completely positive linear operator that acts as a projection onto some subalgebras \cite{Accardi-Cecchini82}).  Then $[y_\beta]\in\partial F^\ast(\beta \Pi_Mx)\subset Y/M^\bot$.  Assume that the corresponding measure $y_\beta(M)=0$.  Then $y_\beta\in[0]\in Y/M^\bot$, where $[0]=M^\bot$, and because $[y_\beta]\geq0$ ($y_\beta\geq0$ and $\Pi_M$ is positive), $[y_\beta]=[0]$ implies by Proposition~\ref{pr:zero}
\[
\Pi_Mx=0\quad\mbox{or}\quad
F^{\ast\ast}([0])=\lambda_0\quad\mbox{or}\quad F^{\ast\ast}([0])=\overline\lambda_M
\]
where $\lambda_0:=\inf F$, and
$\overline\lambda_M\leq\overline\lambda$ is such that $\overline{\Pi_M
  x}(\overline\lambda_M)=\sup\{\langle\Pi_Mx,[y]\rangle:[y]\in\dom
F^{\ast\ast}\}$.  Observe that non-empty $\partial F^{\ast\ast}([0])$
is a singleton set, because $F^\ast$ (and hence $F^\ast(\Pi_Mx)$) is
strictly convex.  Therefore, the last two cases above are false,
because otherwise $\partial F^{\ast\ast}([0])$ would contain the
intervals $[0,\beta \Pi_Mx]$ or $[\beta \Pi_Mx,\infty)$,
  $0<\beta<\infty$.  Thus, $\Pi_Mx=0$ is the only true case.  But then
  $\beta \Pi_Mx=0$ for all $\beta$, and therefore
\[
[0]\in\partial F^\ast(\beta \Pi_Mx)\,,\quad\forall\,\beta\in\bR
\]
In other words, for each $\lambda\in(\lambda_0,\overline\lambda)$, there
is a solution $y_\beta\in[0]$, such that the corresponding measure
$y_\beta(M)=0$.

These measures are not mutually absolutely continuous only if there
exists solution $y^\circ_\beta$ for some
$\lambda\in(\lambda_0,\overline\lambda)$ such that the corresponding
measure $y^\circ_\beta(N)=0$ on some larger subalgebra $N\supset M$.
The subfamily $\{y^\circ_\beta\}_x\subseteq\{y_\beta\}_x$
corresponding to mutually absolutely continuous measures for all
$\lambda\in(\lambda_0,\overline\lambda)$ is constructed by taking
\[
M=\sup\{N\subset X:\exists\,y^\circ_\beta\in\{y_\beta\}_x,\ y^\circ_\beta(N)=0\}
\]
where supremum is with respect to ordering by inclusion.

If $\lambda_0:=\inf F$ (resp. $\overline\upsilon:=\sup\{\langle x,y\rangle:y\in\dom F\}$) is
attained at some $y_0$ (resp. $\delta_x$), then they correspond to
elements of $\{y_\beta\}_x$ with $\beta=0$ (resp. $\beta^{-1}=0$).
The corresponding measures $y_0$ (resp. $\delta_x$) are absolutely
continuous with respect to all $y_\beta^\circ$, because $\Pi_Mx=0$
implies $\beta \Pi_Mx=0$ for all $\beta$.

If $F^{\ast\ast}$ is strictly convex, then $\partial F^\ast(\beta x)$
contains a unique element $y^\circ_\beta$ for each $\beta^{-1}>0$, and
$\{y^\circ_\beta\}_x=\{y_\beta\}_x\setminus\{y_0,\delta_x\}$.
\end{proof}

\begin{remark}
The key condition in the proof of Theorem~\ref{th:absolute} is that
the non-empty subdifferentials $\partial F(y_\beta)$ are singleton
sets, which follows immediately from injectivity of $\partial F^\ast$
or strict convexity of $F^\ast$.  If $y_\beta\in\Int{\dom
  F^{\ast\ast}}$, then $F^{\ast\ast}$ is continuous at $y_\beta$
(e.g. see \cite{Moreau67} or \cite{Rockafellar74}, Theorem~8), and
$\partial F^{\ast\ast}(y_\beta)$ is a singleton if and only if
$F^{\ast\ast}$ is G\^{a}teaux differentiable at $y_\beta$ (e.g. see
\cite{Tikhomirov90:_convex}, Chapter~2, Section~4.1).  Injectivity of
$\partial F^\ast$ can also be based on its algebraic properties.  In
particular, if $\partial F^\ast$ is a group homomorphism, then it is
injective if and only if its kernel is a singleton set.  This will be
discussed in the end of Example~\ref{ex:relative-information} (see
also \cite{Belavkin10:_dis10}).
\end{remark}

Optimal probability measures are obtained by normalization
$p_\beta:=y_\beta/\|y_\beta\|_1$ of optimal positive measures
$y_\beta$.  This corresponds to additional equality $\|y\|_1=\langle
1,y\rangle=1$ and inequality $y\geq0$ constraints in the optimal value
functions~(\ref{eq:support})--(\ref{eq:inverse-anti-support}) or
simply to a restriction of functional $F$ to the statistical manifold
$\mathcal{P}:=\{y:y\geq0,\ \langle 1,y\rangle=1\}$, which is the base
of positive cone $Y_+$.  Optimal probability measures are solutions to
generalized variational problems~(\ref{eq:max-i}) or (\ref{eq:min-i})
with constraints on information distance $I(p,q)$ or resource $F(p)$.
All mutually absolutely continuous measures
$y^\circ_\beta\in\{y_\beta\}_x$ belong to the same subspace
$M^\bot\subset Y$, and the corresponding probability measures
$p^\circ_\beta$ belong to the interior of the base $\cP\cap M^\bot$ of
subcone $M^\bot_+\subset Y_+$.  In the classical (commutative) case,
$\cP$ is a simplex, and $\cP\cap M^\bot$ is its facet, which is itself
a simplex.

\begin{remark}
If the effective domain $\dom F\subset Y$ of functional
$F:Y\rightarrow\bR\cup\{\infty\}$ is the positive cone $Y_+$, then
property $y_\beta(M)=0$ on subalgebra $M\subset X$ implies $y_\beta$
is on the boundary of $Y_+=\dom F$.  In this case, mutual absolute
continuity of measures $y_\beta\in\partial F^\ast(\beta x)$ can be
proved using the fact that the image of injective subdifferential
mapping $\partial F^\ast:X\rightarrow 2^Y$ is interior of $\dom F$
(e.g. see \cite{Alesker98}, Lemma~4).  Therefore, such subgradients
$y_\beta\in\partial F^\ast(\beta x)$ cannot be on the boundary of
$Y_+=\dom F$.
\end{remark}

The existence of optimal and mutually absolutely continuous
probability measures for all constraints $F(y)\leq\lambda$ on an
information resource is used in the next section to study optimality
of deterministic and non-deterministic Markov transition kernels.
Theorem~\ref{th:absolute} shows that this is related to strict
convexity of $F^\ast$ (or injectivity of $\partial F^\ast$), and
therefore we now discuss this property with some examples.

\subsection{Information and separation of variational problems for measures}
\label{sec:strict-convexity}

If $F^\ast$ is not strictly convex (or $\partial F^\ast$ is not
injective), then $\partial F(y_\beta)$ may contain different elements
$x$, $w\in Y^\sharp$.  Recall that linear functionals $x\in Y^\sharp$
are understood in classical optimization theory as objective
(e.g. utility) functions $x:\Omega\rightarrow\bR$ representing a
preference relation $\lesssim$ on $\Omega\equiv\ext\cP$.  Thus,
$y_\beta$ may maximize both $x(y)=\langle x,y\rangle$ and
$w(y)=\langle w,y\rangle$ on $\{y:F(y)\leq\lambda\}$, which means that
$y_\beta$ solves different optimization problems.  Indeed, value
$\lambda=F(y_\beta)$ corresponds to equal optimal values $\overline
x^{-1}(\upsilon)=\overline w^{-1}(\upsilon)$, and value
$\upsilon=\langle x,y_\beta\rangle=\langle w,y_\beta\rangle$ to equal
optimal values $\overline x(\lambda)=\overline w(\lambda)$.
Therefore, if $F^\ast$ is not strictly convex, then elements
$y_\beta\in Y$ may not separate some optimization problems.  Let us
consider two examples.

\begin{example}[Relative information]
Let us define $I_{KL}:Y\times Y\rightarrow\bR\cup\{\infty\}$ as follows
\begin{equation}
I_{KL}(y,y_0):=
\left\{\begin{array}{ll}
\Bigl\langle\ln\frac{y}{y_0},y\Bigr\rangle-\langle 1,y-y_0\rangle &\mbox{ if $y>0$ and $y_0>0$}\\
\langle1,y_0\rangle&\mbox{ if $y=0$ and $y_0>0$}\\
\infty&\mbox{ otherwise}
\end{array}
\right.
\label{eq:information}
\end{equation}
This functional is an extension of the Kullback-Leibler divergence
$\bE_p\{\ln(p/q)\}$ to the whole space $Y$, because $\langle
1,y-y_0\rangle=0$ for positive measures $y$, $y_0$ with equal norms
$\|\cdot\|_1$.  The term $\langle 1,y-y_0\rangle$ makes
$I_{KL}(y,y_0)\geq0$ for all elements $y$ and $y_0$ not necessarily
with equal norms.  If $X$ is a commutative algebra, and the pairing
$\langle\cdot,\cdot\rangle$ is defined by the sum or the
integral~(\ref{eq:pairing}), then (\ref{eq:information}) reduces to
the classical KL-divergence.  In the non-commutative case, such as $X$
being an algebra of compact Hermitian operators and the trace
pairing~(\ref{eq:pairing}), functional~(\ref{eq:information}) is a
generalization of some types of quantum information
\cite{vpb11:_qbic10}, which depend on the way $yy^{-1}_0$ is defined,
such as $\exp(\ln y-\ln y_0)$ or $y_0^{-1/2}yy_0^{-1/2}$.

The functional $F_{KL}(y):=I_{KL}(y,y_0)$ is closed, strictly convex and G\^{a}teaux differentiable on $\Int{\dom F_{KL}}$, and its gradient has the following convenient form:
\[
\nabla F_{KL}(y)=\ln\frac{y}{y_0}\quad\iff\quad
y_0^{1/2}e^x\,y_0^{1/2}=\nabla F_{KL}^\ast(x)
\]
One can define the dual functional $F_{KL}^\ast:X\rightarrow\bR\cup\{\infty\}$ as follows
\[
F_{KL}^\ast(x):=\langle 1,y_0^{1/2}e^{x}\,y_0^{1/2}\rangle
\]
Clearly, $F^\ast_{KL}$ is also closed, strictly convex and G\^{a}teaux
differentiable for all $x\in X$, where it is finite.  Optimal measures
maximizing $x(y)=\langle x,y\rangle$ on sets
$\{y:F_{KL}(y)\leq\lambda\}$ belong to a one-parameter exponential
family $y_\beta:=y_0^{1/2}e^{\beta x}\,y_0^{1/2}$, which are mutually
absolutely continuous.  Such maximizing measures exist for all values
$\lambda\in(\lambda_0,\overline\lambda)$, if $x\in Y^\sharp$ is
$F_{KL}$-bounded above, and by Proposition~\ref{lm:existence} it is
sufficient to show that $\partial F_{KL}^\ast(\beta x)\neq\varnothing$
for some $\beta^{-1}>0$.  We point out that this property depends on
the choice of element $y_0=\nabla F_{KL}^\ast(0)$, minimizing
$F_{KL}$.

Recall also that $Y$ can be considered as a module over algebra $X\subset Y$ (Section~\ref{sec:dual-algebras}).  The exponential mapping $\exp:X\rightarrow X\subset Y$ is the unique (up to the base constant) homomorphism between the additive and multiplicative groups of algebra $X$, and it is injective, because it has a singleton kernel $\{x:\exp(x)=yy^{-1}=1\}=\{0\}$.  The property $\nabla F_{KL}(y)=\ln(yy_0^{-1})=(\exp)^{-1}(yy_0^{-1})$ ensures that information distance $I_{KL}(y,y_0)=F_{KL}(y)$ is additive:  $I_{KL}(p_1p_2,q_1q_2)=I_{KL}(p_1,q_1)+I_{KL}(p_2,q_2)$ for all $p_1p_2$, $q_1q_2\in\cP$.
\label{ex:relative-information}
\end{example}

\begin{figure}[ht]
\begin{center}
\input{simplex-l1}
\end{center}
\caption{$2$-Simplex $\cP$ of probability measures over set $\Omega=\{\omega_1,\omega_2,\omega_3\}$ with level sets of expected utilities $\bE_p\{x\}=\bE_p\{w\}=\upsilon$ and the total variation metric $\|p-q\|_1=\lambda$.  Probability measure $p_\beta$ maximizes both $\bE_p\{x\}$ and $\bE_p\{w\}$ subject to constraint $\|p-q\|_1\leq\lambda$.  The family $\{p_\beta\}_x$ of solutions, shown by dashed line, contains elements on the boundary of $\cP$.}
\label{fig:simplex-l1}
\end{figure}
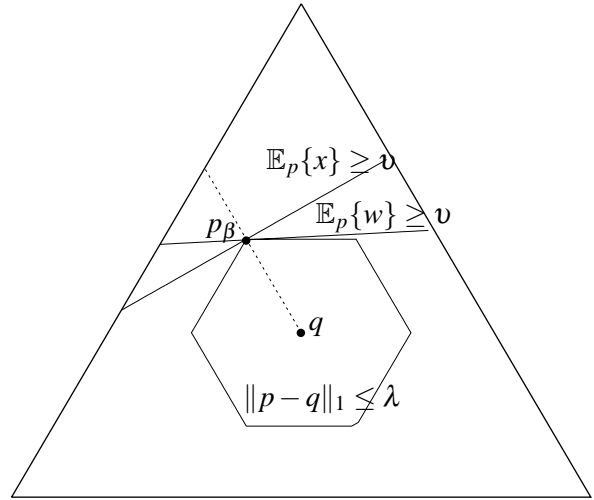

\begin{example}[Total variation]
An example of information distance that does not have a strictly convex dual is the total variation metric:
\[
I_V(y,y_0):=\|y-y_0\|_1
\]
Functional $F_V(y):=I_V(y,y_0)$ is not G\^{a}teaux differentiable at $y=y_0$, as well as $y$ such that $y-y_0\in[0]\in Y/M^\bot$, if subalgebra $M\subset X$ bounds $X_+$ (e.g. if $M$ contains an extreme ray of $X_+$).   Optimal solutions $y_\beta$ maximizing $x(y)=\langle x,y\rangle$ on sets $C(\lambda):=\{y:\|y-y_0\|_1\leq\lambda\}$ are extreme points of $C(\lambda)$, and they maximize different, not necessarily proportional linear functionals.  Figure~\ref{fig:simplex-l1} illustrates the variational problems on a $2$-simplex of probability measures over a set of three elements with the uniform distribution $q(\omega)=1/3$ as the reference measure (compare with Figure~\ref{fig:simplex-kl}).  Distribution $p_\beta$ maximizes both $\bE_p\{x\}=\langle x,p\rangle$ and $\bE_p\{w\}=\langle w,p\rangle$ on $C(\lambda):=\{p:\|p-q\|_1\leq\lambda\}$.

The dual of $F_V$ is functional $F_V^\ast(x)=\chi_{C_0^\circ(\lambda)}(x)-\langle x,y_0\rangle$, where $\chi_{C_0^\circ(\lambda)}(x)$ is the indicator function of set $C_0^\circ(\lambda)=\{\beta x:\|\beta x\|_\infty\leq1\}$, the polar of set $C_0(\lambda)=C(\lambda)-y_0$.  Clearly, $F_V^\ast(x)$ is not strictly convex.  Therefore, $\partial F_V(y_\beta)$ may include multiple elements, and the family $\{y_\beta\}_x$ may contain measures that are not mutually absolutely continuous.  Figure~\ref{fig:simplex-l1} shows that the family $\{p_\beta\}_x$ of optimal solutions contains elements on the boundary of $2$-simplex $\cP$.

In the commutative case, elements of $\partial F_V(y_\beta)\subset X$ are understood as utility functions, representing preference relations $\lesssim$ on $\Omega\equiv\ext\cP$.  If $\partial F_V(y_\beta)$ includes functions $x$ and $w$, then they attain their suprema $\sup x(\omega)=x(\top)=\|x\|_\infty$ and $\sup w(\omega)=w(\top)=\|w\|_\infty$ on the set of the same elements $\top\in\Omega$.   However, the utility functions $x(\omega)$ and $w(\omega)$ may represent different preference relations $\lesssim$ on $\Omega$.  Note also that the suprema $x(\top)$ or $w(\top)$ of utilities may never be achieved or observed in problems with constraints on information, even if $x$ or $w$ are bounded functions.  The values of utilities on elements $\omega\neq\top$ are important for maximization of the expected utility.
\end{example}

As was discussed in Section~\ref{sec:axioms}, information is often required to satisfy the additivity axiom, which is why information-theoretic definitions of entropy and mutual information are based on the KL-divergence $I_{KL}(y,y_0)$, and it has a strictly convex dual.  Strict convexity of the dual functional is a weaker condition than the additivity axiom, but it ensures that each probability measure $p\in\cP$ is an optimal solution to a unique variational problem with an abstract information resource $F$, generalizing problems~(\ref{eq:max-i}) or (\ref{eq:min-i}).  Note also that strict convexity of $F^\ast$ ensures that information resource $F$ has directional derivative at each $y\in\Int{\dom F}$ (e.g. $p\in\Int{\cP}$), which facilitates convergence of measures in problems with dynamic information.  Thus, strict convexity of the dual functional appears to be a natural requirement on the functional representing information.

\subsection{Support of utility functions and operators}

We now conclude this section by the following corollary about the support of utility functions or operators.  We remind that the support of function $x:\Omega\rightarrow\bR$ is the set $\mathrm{supp}(x):=\{\omega:x(\omega)\neq0\}$.  The support of an operator $x$ on a Hilbert space is defined as a projection onto the orthogonal complement of its kernel (e.g. \cite{Dixmier81}, Appendix~III).  When $x$ is considered as an element of algebra $X$, its restriction to a subset $E\subset\Omega$ (subspace $E\subset\mathcal{H}$) is given by localization $\Pi_Mx$ of $x$ onto subalgebra $M\subset X$ corresponding to $E$.  Thus, the support of $x$ can be identified with the complement of the largest subalgebra $M\subset X$ such that $\Pi_Mx=0$.

\begin{corollary}[Support]
Under the assumptions of Theorem~\ref{th:absolute}, the support of
element $x\in X$ is a subset of the support of optimal measures
$y_\beta$ for all $\lambda\in(\lambda_0,\overline\lambda)$.
\label{cr:support}
\end{corollary}

\begin{proof}
During the proof of Theorem~\ref{th:absolute}, we established under
its assumptions, that if solution $y_\beta(M)=0$ for some
$\lambda\in(\lambda_0,\overline\lambda)$ and $M\subset X$, then the localization
$\Pi_Mx=0$.  Dually, if $\Pi_Mx\neq0$ for some $M\subset X$, then
$y_\beta(M)\neq0$ for all such $y_\beta$.
\end{proof}

Because random variables or observables are considered with respect to normalized positive measures (i.e. probability measures), they can be treated not as elements of algebra $X$, dual of $Y$, but as elements of the factor space $X/\bR1$, generated by subspace $\bR1:=\{\beta1:\beta\in\bR,\ 1\in X\}$ of scalar vectors.  Indeed, statistical manifold $\mathcal{P}$ is a subset of the affine set $\{y:\langle1,y\rangle=1\}=\{1\}_\bot+q$, where $\{1\}_\bot$ is the annihilator of element $1\in X$, and $q\in\mathcal{P}$.  Thus, every probability measure $p\in\mathcal{P}$ is equivalently represented by elements $y\in\{1\}_\bot$ as $p=y+q$.  The dual of subspace $\{1\}_\bot$ is the factor space $X/\bR1$, and random variables are affine sets $[x]=\bR1+x$ corresponding to equivalence classes $[x]=\{w:x-w\in\bR1\}$ and $\langle x-w,p-q\rangle=0$ for any $p$, $q\in\mathcal{P}$.  Observe now that $\bR1$ is the zero element in $X/\bR1$, and therefore the fact that localization $\Pi_Mx\notin\bR1$ implies $p_\beta(M)>0$ for all optimal probability measures (Corollary~\ref{cr:support}).  Dually, $p_\beta(M)=0$ implies that $\Pi_Mx\in\bR1$.  In the language of classical probability this can be stated as follows: if $x(\omega_1)\neq x(\omega_2)$ for some $\omega_1$, $\omega_2\in E\subset\Omega$, then $p_\beta(E)>0$ for all probability measures maximizing $\bE_p\{x\}$ on sets $\{p:F(p)\leq\lambda\}$ for all $\lambda\in(\lambda_0,\overline\lambda)$.  Dually, $p_\beta(E)=0$ implies that $x(\omega)=\mathrm{const}$ for all $\omega\in E$.

\section{Optimal Markov transition kernels}
\label{sec:kernels}

In this section, we consider a composite system, such as a direct product $\Omega=A\times B$ of two sets, and the problem of optimization of transitions between the elements of $A$ and $B$.  Such problems appear in theories of decisions, control, communication and computation, where components of a system (represented by sets $A$, $B$, etc) may have different meanings, but the main objective is to find transitions between the elements of $A$ and $B$ that are optimal with respect to a utility function $x:A\times B\rightarrow\bR$.  In some cases, optimal transitions are deterministic corresponding to some functions $a=f(b)$ or $b\in f^{-1}(a)$.  More generally, non-deterministic transitions are represented by conditional probabilities or Markov transition kernels.  For simplicity, our exposition will be in the classical setting of commutative algebra $X:=C_c(\Omega,\bR,\|\cdot\|_\infty)$ of functions on $\Omega=A\times B$.  This is because joint and conditional probabilities are well-defined and understood in this setting.  In the non-classical case, the analogue of a conditional probability operator can also be defined (e.g. \cite{Accardi-Cecchini82,Petz88:_cond,Takesaki72}), and the results of this section can then be transferred to this setting.  However, this leads to unnecessary complications, which we shall avoid.

\subsection{Markov transition kernels and information constraints}

Let us remind the following definition (e.g. see \cite{Chentsov72}, Sections~2 and 5).

\begin{definition}[Markov transition kernel]
Given two measurable sets $(A,\mathcal{A})$ and $(B,\mathcal{B})$, a {\em Markov transition kernel} is a conditional probability measure $P(A_i\mid b)\in\cP(A)$ on $(A,\mathcal{A})$, which is $\mathcal{B}$-measurable for each $A_i\in\mathcal{A}$.
\label{def:kernel}
\end{definition}

Markov transition kernel defines linear transformation $\Pi:\mathcal{P}(B)\rightarrow\mathcal{P}(A)$ between statistical manifolds $\mathcal{P}(A)$ and $\mathcal{P}(B)$ as follows:
\[
P(A_i)=\Pi P(B_j):=\int_{B_j} P(A_i\mid b)\,dP(b)
\]
Elements $p\in\mathcal{P}(A\times B)$ are joint probability measures
$P(A_i\times B_j)=P(A_i\mid B_j)\,P(B_j)$, and for $P(B_j)>0$, the conditional
probability is defined by the Bayes formula:
\[
P(A_i\mid B_j)=\frac{P(A_i\times B_j)}{P(B_j)}\,,
\]
Event $a\in A$ is statistically independent of $b\in B$ if and only if $P(A_i\mid b)=P(A_i)$ for each $b\in B$ and all $A_i\in\mathcal{A}$.  In this case, $P(A_i\times B_j)=P(A_i)P(B_j)$.  On the other hand, a function $a=f(b)$ defines deterministic dependency of $a$ on $b$, and it corresponds to a deterministic transition kernel
\[
P(A_i\mid b)=\delta_{f(b)}(A_i):=\left\{
\begin{array}{cl}
1 & \mbox{ if $f(b)\in A_i$}\\
0 & \mbox{ otherwise}
\end{array}\right.
\]

One can see that each joint probability measure $p\in\mathcal{P}(A\times B)$ defines a pair of marginal and conditional probability measures $P(B)$ and $P(A\mid B)$ or $P(A)$ and $P(B\mid A)$.  Thus, points of $\cP(A\times B)$ define all possible transition kernels, including all possible measurable functions between $A$ and $B$.  Hence the following classification.

\begin{definition}[Deterministic composite state]
A joint probability measure $p\in\mathcal{P}(A\times B)$ is {\em deterministic}, if and only if it defines a deterministic transition kernel $\delta_{f(b)}(A_i)$ for some measurable function $f:B\rightarrow A$ or $f^{-1}:A\rightarrow B$.  Otherwise, $p$ is {\em non-deterministic}.
\label{def:deterministic}
\end{definition}

Transition kernels are often understood as communication channels giving a more traditional meaning to the notion of information related to the process of sending messages between $A$ and $B$.  The amount of information communicated by $P(A_i\mid b)$ is measured by the Shannon
mutual information \cite{Shannon48}:
\begin{equation}
I_S\{a,b\}:=\int_{A\times B}\left[\ln\frac{dP(a,b)}{dP(a)\,dP(b)}\right]\,dP(a,b)
=\int_B dP(b)\int_A\left[\ln\frac{dP(a\mid b)}{dP(a)}\right]\,dP(a\mid b)
\label{eq:Shannon}
\end{equation}
One can see that $I_S\{a,b\}$ is defined as information distance $I_{KL}(p,q):=\bE_p\{\ln(p/q)\}$ of joint measure $p:=P(A_i\times B_j)$ from the product of marginals $q:=P(A_i)\,P(B_j)$, or as the expectation of the information distance $I_{KL}$ of the conditional probability $P(A_i\mid b)$ from the marginal $P(A_i)$, taken with respect to a fixed marginal $P(B_j)$.

Variational problems~(\ref{eq:max-i}) and (\ref{eq:min-i}) for composite systems and constraints on mutual information have been studied in information theory (e.g. \cite{Shannon48,Stratonovich65,Stratonovich75:_inf}).  Note that when problems~(\ref{eq:max-i}) and (\ref{eq:min-i}) are considered on any measurable set $\Omega$, they are referred to in information theory as problems of the first kind \cite{Stratonovich75:_inf}.  For a composite system $\Omega=A\times B$, one distinguishes between problems of the second and third kind.  Observe that the amount of mutual information~(\ref{eq:Shannon}) communicated depends on $P(B_j)$, which we refer to as an the input or source distribution, and transition probabilities $P(A_i\mid b)$.  In fact, $I_S\{a,b\}=H\{b\}-H\{b\mid a\}$, where $H\{b\}:=\bE_p\{-\ln P(b)\}$ is the entropy of $P(B)$, and $H\{b\mid a\}$ is the conditional entropy.  Optimization problems over input distributions $P(B)$ and with a fixed channel $P(A_i\mid b)$ are problems of the second kind.  Problems of the third kind are concerned with finding an optimal channel for a fixed set of input distributions.  The results of previous sections allow us to consider a generalization of these problems when mutual information is defined by some other information distance $I(p,q)$ between two joint states $p$, $q\in\cP(A\times B)$ or an information resource $F(p)$.  Note that problems of the third kind play important role not only in information theory, but also in other areas including optimal statistical decisions, estimation, control and even in the theory of algorithms, as will be illustrated in Section~\ref{sec:algorithms}.

\subsection{Strict sub-optimality of deterministic kernels}

Observe that $P_f(A_i\times B_j)=\delta_{f(b)}(A_i)\,P(B_j)=0$ for all $f(b)\notin A_i$.  Thus, deterministic transition kernels can be defined only by joint states that are on the boundary of $\cP(A\times B)$; interior points of $\cP(A\times B)$ can define only non-deterministic transition kernels.  The application of Theorem~\ref{th:absolute} to the case $\Omega=A\times B$ yields the following result.

\begin{theorem}[Separation of deterministic and non-deterministic kernels]
Let $\{p_\beta\}_x\subset\mathcal{P}(A\times B)$ be a family of joint
probability measures maximizing expected value $\bE_p\{x\}=\langle x,p\rangle$ of function $x:A\times B\rightarrow\bR$ on sets $\{p:F(p)\leq\lambda\}$ for all values $\lambda$ of a closed functional $F:\cP\rightarrow\bR\cup\{\infty\}$.  If $F^\ast(x):=\sup\{\langle x,p\rangle-F(p)\}$ is strictly convex and $F$ is minimized at $p_0\in\partial F^\ast(0)\subset\Int{\mathcal{P}(A\times B)}$, then
\begin{enumerate}
\item $\{p_\beta\}_x$ contains deterministic $p_f$ if and only if it is a solution to
  an unconstrained problem: $\lambda\geq\overline\lambda$ or $\langle
  x,p_f\rangle=\overline\upsilon:=\overline x(\overline\lambda)=\sup\{\langle x,p\rangle:p\in\cP(A\times B)\}$.
\item The inequality
\[
\langle x,p_f\rangle<\langle x,p_\beta\rangle
\]
holds for all deterministic $p_f\in\mathcal{P}(A\times B)$ such that $F(p_f)=F(p_\beta)\in(\lambda_0,\overline\lambda)$.
\item Similarly, the inequality
\[
F(p_f)>F(p_\beta)
\]
holds for all deterministic $p_f\in\mathcal{P}(A\times B)$ such that $\langle x,p_f\rangle=\langle x,p_\beta\rangle\in(\overline\upsilon_0,\overline\upsilon)$.
\end{enumerate}
\label{th:nofunctions}
\end{theorem}

\begin{proof}
\begin{enumerate}
\item ($\Rightarrow$) Assume there exists $p_f\in\{p_\beta\}_x$ for
  $\lambda<\overline\lambda$ (and $\langle
  x,p_f\rangle<\overline\upsilon$), and such that the corresponding
  transition kernel is deterministic: $P_f(A_i\mid B_j)=1$ if
  $A_i=f(B_j)$ and $P_f(A\setminus A_i\mid B_j)=0$.  In this case,
  $p_f:=P_f(A\times B)$ is not in the interior of $\mathcal{P}(A\times
  B)$, because $P_f((A\setminus f(B_j))\times B_j)=0$, and in particular
  $p_f$ does not minimize $F$, because $\partial
  F^\ast(0)\subset\Int{\cP(A\times B)}$ by our assumption.  Thus,
  $F(p_f)=\lambda\in(\lambda_0,\overline\lambda)$.  But then
  $P_f((A\setminus f(B_j))\times B_j)=0$ implies that there exist
  $p_\beta^\circ\in\{p_\beta\}_x$ for all $\lambda\in[\lambda_0,\infty]$
  such that $p_\beta^\circ:=P_\beta^\circ((A\setminus f(B_j))\times
  B_j)=0$ by Theorem~\ref{th:absolute}.  In particular, there exists
  $p_0^\circ\in\partial F^\ast(0)$ such that $P_0^\circ((A\setminus
  f(B_j))\times B_j)=0$, and therefore $p_0^\circ$ is also not in the
  interior of $\mathcal{P}(A\times B)$.  Thus, by contradiction we
  have proven $p_f\notin\{p_\beta\}_x$ or
  $\lambda\geq\overline\lambda$ (and hence $\langle
  x,p_f\rangle=\overline\upsilon$).

  ($\Leftarrow$) If $\lambda\geq\overline\lambda$, then there exists solution
  $\delta_x\in\ext\mathcal{P}(A\times B)$ such that $\langle
  x,\delta_x\rangle=\overline\upsilon:=\sup\{\langle x,p\rangle:p\in\cP\}$ (by linearity of $\langle
  x,\cdot\rangle$ and Krein-Milman theorem for $\cP$), and $\delta_x$
  corresponds to some function $f(b)=a$.
\item For all $x\in X$ and $y\in Y$, the Young-Fenchel inequality
  holds: $\langle x,y\rangle\leq F^\ast(x)+F(y)$.  Moreover, it holds
  with equality if and only if $y\in\partial F^\ast(x)$ (e.g. see
  \cite{Tikhomirov90:_convex}, Chapter~2, Section~4.1, Lemma~3).
  Assume $p_\beta\in\partial F^\ast(\beta x)$.  Then $\langle
  x,p_\beta\rangle=\beta^{-1}[F^\ast(\beta x)+F(p_\beta)]$.  On the
  other hand, if $p_f$ is deterministic and $F(p_f)\leq\lambda<\overline\lambda$, then $p_f\notin\partial F^\ast(\beta x)$ and therefore
\[
\langle x,p_f\rangle<\beta^{-1}[F^\ast(\beta
  x)+F(p_f)]=\beta^{-1}[F^\ast(\beta
  x)+F(p_\beta)]=\langle x,p_\beta\rangle
\]
\item By definition of the Legendre-Fenchel transform,
  $F^{\ast\ast}(y)\geq\langle x,y\rangle-F^\ast(x)$, and the equality
  holds if and only if $x\in\partial F^{\ast\ast}(y)$.  Assume $\beta
  x\in\partial F^{\ast\ast}(p_\beta)$.  Then
  $F^{\ast\ast}(p_\beta)=F(p_\beta)=\beta\langle
  x,p_\beta\rangle-F^\ast(\beta x)$.  On the other hand, if $p_f$ is
  deterministic and $\langle x,p_f\rangle<\overline\upsilon$, then
  $\beta x\notin\partial F^{\ast\ast}(p_f)$, and therefore
\[
F(p_f)\geq F^{\ast\ast}(p_f)>\beta\langle x,p_f\rangle-F^\ast(\beta x)=\beta\langle x,p_\beta\rangle-F^\ast(\beta x)=F(p_\beta)
\]
Note that $\beta>0$ and $F(p_\beta)=\lambda>\lambda_0$, if $\langle x,p_\beta\rangle=\upsilon>\overline\upsilon_0$.
\end{enumerate}
\end{proof}

The assumptions of Theorem~\ref{th:nofunctions} are quite general.
The relation of strict convexity of $F^\ast$ to separating property of
information of variational problems for measures was discussed in
Section~\ref{sec:strict-convexity}.  The assumption
$p_0\in\Int{\mathcal{P}(A\times B)}$ is very natural.  Indeed, each
facet of the simplex $\mathcal{P}(A\times B)$ is also a simplex of
some subset of $A\times B$.  Therefore, the element $p_0$ is always in
the interior of some simplex $\mathcal{P}(A_i\times B_j)$, unless
$p_0=\delta\in\ext\mathcal{P}(A\times B)$.  In all practical cases,
information is minimized at $p_0\notin\ext\mathcal{P}(A\times B)$.  In
particular, one often chooses $p_0:=P(A_i)P(B_j)$, so that $a$ and $b$
are independent, and supports of marginal probabilities $P(A_i)$ and $P(B_j)$ include
more than one element.

To understand better the result of Theorem~\ref{th:nofunctions}, we now recall some facts about mutual information for deterministic kernels and then for exponential kernels, which are an important example of non-deterministic kernels.  These facts will be used in a qualitative example, presented later.

\subsection{Deterministic transition kernels}

Probability measure $P(A_i)=\Pi_fP(B_j)$ defined by a linear transformation with deterministic transition kernel $\delta_{f(b)}(A_i)$ is sometimes denoted $Pf^{-1}(A_i):=P\{b:f(b)\in A_i\}$ (e.g. \cite{Chentsov72}, Section~2).  If $f:B\rightarrow A$ is injective, then $Pf^{-1}(A_i)=P(B_j)$ for each $A_i=f(B_j)$.

\begin{definition}[Measurable isomorphism]
An injective and measurable function $f:B\rightarrow A$ is called a {\em measurable monomorphism} of $B$.  If $f$ is also surjective and $f^{-1}(a)$ is measurable, then $f$ is a {\em measurable isomorphism}.
\label{def:isomorphism}
\end{definition}

We point out the following known result.

\begin{proposition}[Invertible transformation]
A linear transformation $\Pi:\cP(B)\rightarrow\cP(A)$ of statistical manifolds is invertible if and only if its Markov transition kernel is $\delta_{f(b)}(A_i)$, where $f$ is a measurable isomorphism.
\label{pr:invertible}
\end{proposition}

\begin{proof}
($\Rightarrow$)
Assume that the transition kernel of $\Pi$ is not defined by any function.  Thus, $\Pi\delta_b=p\notin\ext\cP(A)$ for some $\delta_b\in\ext\cP(B)$.  Without loss of generality, we can assume that $p=(1-t)\delta_{a_1}+t\delta_{a_2}$ for some $t\in(0,1)$, $\delta_{a_1}$, $\delta_{a_2}\in\ext\cP(A)$ such that $\delta_{a_1}\neq\delta_{a_2}$.  Then
\[
\Pi^{-1}p=\Pi^{-1}[(1-t)\delta_{a_1}+t\delta_{a_2}]=(1-t)\Pi^{-1}\delta_{a_1}+t\Pi^{-1}\delta_{a_2}=\delta_b
\]
Because $\delta_b\in\ext\cP(B)$ is not a convex combination of any points of $\cP(B)$, it implies $\Pi^{-1}\delta_{a_1}=\Pi^{-1}\delta_{a_2}=\delta_b$.  But then $\Pi^{-1}$ is not injective, because $\delta_{a_1}\neq\delta_{a_2}$, and therefore $\Pi$ is not surjective.  Thus, the transition kernel of an invertible $\Pi$ must be $\delta_{f(b)}(A_i)$ for some measurable function $f:B\rightarrow A$.   Clearly, such $\Pi$ is invertible only if the mapping $f:\ext\cP(B)\rightarrow\ext\cP(A)$ is injective, surjective, and both $f$ and $f^{-1}$ are measurable.

($\Leftarrow$)  Obvious.
\end{proof}

Let us consider information communicated by a deterministic transition kernel $\delta_{f(b)}(A_i)$.  The maximum (or supremum) amount of information can be communicated if $f$ is an injective function, because preimage $f^{-1}(a)$ uniquely determines $b$.  If a function is not injective, then $b\in f^{-1}(a)$ is determined up to the probability $1/|f^{-1}(a)|$.  Indeed, for countable $B$ and constant $P(b)$\footnote{The condition $P(b)=\mathrm{const}$ was omitted in the final version.} this can be shown as follows:
\[
P_f(b\mid a)=\frac{P_f(a,b)}{P_f(a)}
=\frac{\delta_{f(b)}(a)\,P(b)}{\sum_B \delta_{f(b)}(a)\,P(b)}
=\frac{1\cdot P(b)}{\sum_{b\in f^{-1}(a)}1\cdot P(b)}
=\frac{1}{|f^{-1}(a)|}
\]
We can express the average amount of information communicated by function $f$ by the following {\em injectivity index} of $f$:
\[
I(f):=\frac{1}{\bE\{|f^{-1}(a)|\}}\leq1
\]
Note that if $B$ is finite, then we can compute the injectivity index as $I(f)=|f(B)|/|B|$.  Indeed, $\sum_{a\in f(B)}|f^{-1}(a)|=|B|$, and so the average value of $|f^{-1}(a)|$ is $|B|/|f(B)|$.  Thus, $I(f)=1$ for an injective function, and $\inf I(f)=0$ corresponding to an empty function.  For constant functions, $I(f)=1/|B|$, and they communicate the least amount of information among non-empty functions.  If $B$ is finite, then $I(f)<1$ implies $|f(B)|<|B|$.  This is not the case, however, for functions defined on an infinite set (e.g. $I(f)=1/2$ for $f:\mathbb{Z}\rightarrow\bN$ defined as $f(b)=|b|$, but $|f(B)|=|B|=\aleph_0$).  Let us show that if the image of a function is infinite, then one can always construct an input distribution $P(B)$ such that the output distribution $Pf^{-1}(A)$ has infinite entropy.

\begin{proposition}[Maximizing input distribution]
Let $(A,\mathcal{A})$ and $(B,\mathcal{B})$ be infinite measurable sets, and let $\{f_n\}$ be a sequence of measurable functions $f_n:B\rightarrow A$ with finite images.  There exists a sequence of probability measures $P_n$ on $\mathcal{B}$ such that
\[
\lim_{|f_n(B)|\rightarrow\infty}\left\{H_n\{a\}=-\sum_{a\in f_n(B)}\ln[P_nf_n^{-1}(a)]\,P_nf_n^{-1}(a)\right\}=\infty
\]
\label{pr:maximising-input}
\end{proposition}

\begin{proof}
It is sufficient to take $P_n$ on $B$ that induce under the mappings $f_n:B\rightarrow A$ constant (i.e. uniform) probability distributions on the images $f_n(B)$.  For example, assuming without loss of generality that $B$ is countable, define the following function on $B$:
\[
P_n(b)=\frac{1}{|f_n(B)|}\frac{1}{|f_n^{-1}\circ f_n(b)|}
\]
It is a probability measure, because it is positive, additive and $P_n(B)=1$.  Indeed
\[
P_n(B_j)=\frac{1}{|f_n(B)|}\sum_{b\in B_j}\frac{1}{|f_n^{-1}\circ f_n(b)|}
\leq\frac{1}{|f_n(B)|}\sum_{a\in f_n(B_j)} \frac{|f_n^{-1}(a)|}{|f_n^{-1}(a)|}
=\frac{|f_n(B_j)|}{|f_n(B)|}
\]
where equality holds if and only if $B_j=f_n^{-1}\circ f_n(B_j)$.  Then
\[
P_nf_n^{-1}(a)
=\frac{1}{|f_n(B)|}\sum_{b\in f_n^{-1}(a)}\frac{1}{|f_n^{-1}\circ f_n(b)|}
=\frac{1}{|f_n(B)|}\frac{|f_n^{-1}(a)|}{|f_n^{-1}(a)|}
=\frac{1}{|f_n(B)|}
\]
The entropy of $P_nf_n^{-1}(a)$ is $H_n\{a\}=\ln|f_n(B)|$, and it grows infinitely with $|f_n(B)|$.
\end{proof}

It follows from Proposition~\ref{pr:maximising-input} that if the amount of information communicated by a deterministic transition kernel $\delta_{f(b)}(A_i)$ is finite for any input distribution $P(B_j)$, then the image of $f$ must be finite.  Note that this argument is not based on any specific notion of mutual information.  For Shannon information, one can show that the following inequality holds for a deterministic kernel $\delta_{f(b)}(A_i)$:
\begin{eqnarray}
I_S\{a,b\}&=&\sum_{b\in B} P(b)\sum_{a\in A}\left[\ln\frac{\delta_{f(b)}(a)}{Pf^{-1}(a)}\right]\,\delta_{f(b)}(a)\nonumber\\
&=&\sum_{b\in B}P(b)\,\left[\ln\frac{1}{Pf^{-1}\circ f(b)}\right]\leq\ln|f(B)|
\label{eq:shannon-deterministic}
\end{eqnarray}
This inequality is obtained by maximizing $I_S\{a,b\}$ for a fixed deterministic kernel $\delta_{f(b)}(A_i)$ over all input distributions $P(b)$.  The supremum of $I_S\{a,b\}$ is achieved at $P(b)$ inducing a constant distribution $Pf^{-1}(a)$ on $A$, such as the maximizing distribution in Proposition~\ref{pr:maximising-input}.

\subsection{Exponential kernels}
\label{sec:exponential-kernels}

If the function $f:B\rightarrow A$ is not injective, then there exist input distributions $P(B)$ with non-zero entropy such that $Pf^{-1}(a)=1$ for some $a\in A$.  In this case, the output entropy $H\{a\}$ is zero, and the transition kernel communicates no information.  Moreover, if $f:B\rightarrow A$ has infinite domain and finite image, then its injectivity index is zero: $\lim_{|B|\rightarrow\infty}|f(B)|/|B|=0$.  This means that such a function can potentially `loose' an infinite amount of information.  Non-deterministic transition kernels, on the other hand, are quite different in this sense, because there exist kernels that always communicate some information.  An important example are exponential transition kernels.

Let $\Omega=A\times B$ and $x:A\times B\rightarrow\bR$ be a utility function.  Consider variational problems~(\ref{eq:max-i}) and (\ref{eq:min-i}) with $I_{KL}(p,q):=\bE_p\{\ln[p/q]\}$ defining Shannon mutual information~(\ref{eq:Shannon}).  The unique solutions to these problems are joint probability measures $p_\beta\in\mathcal{P}(A\times B)$ that belong to a one-parameter exponential family:
\[
dP_\beta(a,b)=e^{\beta\,[x(a,b)+\Phi(\beta^{-1})]}\,dP(a)\,dP(b)\,,
\]
where $\Phi(\beta^{-1})$ is determined from the normalization
condition
\[
e^{-\beta\,\Phi(\beta^{-1})}=\int_{A\times B}e^{\beta\,x(a,b)}\,dP(a)\,dP(b)
\]
The corresponding exponential transition kernels are
\[
dP_\beta(a\mid b)=e^{\beta\,[x(a,b)+\Phi(\beta^{-1},b)]}\,dP(a)\,,\qquad
dP_\beta(b\mid a)=e^{\beta\,[x(a,b)+\Phi(\beta^{-1},a)]}\,dP(b)
\]
where $\Phi(\beta^{-1},b)$ and $\Phi(\beta^{-1},a)$ now depend on $b$
and $a$, as they are computed using partial integrals:
\[
e^{-\beta\,\Phi(\beta^{-1},b)}=\int_A e^{\beta\,x(a,b)}\,dP(a)\,,\quad
e^{-\beta\,\Phi(\beta^{-1},a)}=\int_B e^{\beta\,x(a,b)}\,dP(b)
\]

If the product $e^{\beta\,\Phi(\beta^{-1},b)}\,dP(b)$ does not depend
on $b$, and $e^{\beta\,\Phi(\beta^{-1},a)}\,dP(a)$ does not depend on
$a$, then exponential kernels do not depend on the marginal measures
$dP(a)$ and $dP(b)$ respectively.  Indeed, because
$dP(a)=\int_BdP(a,b)$ and $dP(b)=\int_AdP(a,b)$, we have the following
equations
\[
\int_B e^{\beta[x(a,b)+\Phi(\beta^{-1},b)]}\,dP(b)=1\,,\qquad
\int_A e^{\beta[x(a,b)+\Phi(\beta^{-1},a)]}\,dP(a)=1
\]
Then, using the facts that $e^{\beta\,\Phi(\beta^{-1},b)}\,dP(b)$ and $e^{\beta\,\Phi(\beta^{-1},a)]}\,dP(a)$ are constants, we obtain:
\[
e^{-\beta\,\Phi(\beta^{-1},b)}=[dP(b)/db]\int_B e^{\beta\,x(a,b)}\,db\,,\quad
e^{-\beta\,\Phi(\beta^{-1},a)}=[dP(a)/da]\int_A e^{\beta\,x(a,b)}\,da
\]
Using these relations and the Bayes formula the exponential transition kernels can be written in the following simple form
\[
dP_\beta(a\mid b)=\frac{e^{\beta\,x(a,b)}\,da}{\int_A e^{\beta\,x(a,b)}\,da}\,,\qquad
dP_\beta(b\mid a)=\frac{e^{\beta\,x(a,b)}\,db}{\int_B e^{\beta\,x(a,b)}\,db}
\]
Here, the normalizing integrals are constant, because they do not depend on
$a$ or $b$, and one can introduce the {\em free energy} function
$\Phi_0(\beta^{-1}):=-\beta^{-1}\ln\int_B e^{\beta\,x(a,b)}\,db$ or
the {\em free cumulant generating function}
$\Psi_0(\beta)=-\beta\Phi_0(\beta^{-1})$.  If one of the marginal
distributions, say $P(B)$, is fixed, then Shannon information has the
following expression:
\begin{eqnarray}
I_S\{a,b\}&=&\int_A dP(a)\int_B\left[\ln\frac{dP(b\mid a)}{dP(b)}\right]\,dP(b\mid a)\nonumber\\
&=&\int_A dP(a)\int_B\Bigl\{\beta\,x(a,b)-\ln\int_B e^{\beta\,x(a,b)}\,db-\ln[dP(b)/db]\Bigr\}\,dP(b\mid a)\nonumber\\
&=&\beta\,\bE_{p_\beta}\{x\}-\Psi_0(\beta)+H\{b\}\,,\label{eq:i-exp-kernel}
\end{eqnarray}
Observe also that the expected utility is the derivative of
$\Psi_0(\beta)=\ln\int_B e^{\beta\,x(a,b)}\,db$:
\begin{equation}
\bE_{p_\beta}\{x\}=\int_A dP(a)\int_B\frac{x(a,b)\,e^{\beta\,x(a,b)}}{\int_B e^{\beta\,x(a,b)}\,db}\,db=\frac{d\Psi_0(\beta)}{d\beta}\int_AdP(a)=\Psi'_0(\beta)
\label{eq:u-exp-kernel}
\end{equation}
Here, $H\{b\}=-\int_B\ln[dP(b)/db]\,dP(b)$ is the differential entropy
of $P(B)$ (assuming that the density $dP(b)/db$ exists).  Also,
because $I_S\{a,b\}=H\{b\}-H\{b\mid a\}$, the difference
$\Psi_0(\beta)-\beta\,\Psi'_0(\beta)$ is the conditional differential
entropy $H\{b\mid a\}$.  Expected utility defined by
equation~(\ref{eq:u-exp-kernel}) is independent of the input
distribution $P(B)$.

One can show that the products $e^{\beta\,\Phi(\beta^{-1},b)}\,dP(b)$ and $e^{\beta\,\Phi(\beta^{-1},a)}\,dP(a)$ are constant when $A=(A,+)$ and $B=(B,+)$ are equivalent locally compact groups with invariant measures $da$ and $db$, and the utility function is translation invariant: $x(a+c,b+c)=x(a,b)$.  An important example is when $A$ and $B$ are equivalent linear spaces, and $x(a,b)$ depends only on the difference $a-b$ (e.g. $x(a,b)=-\frac12\|a-b\|^2$).  In such cases, the simplified expressions and equations~(\ref{eq:i-exp-kernel}) and (\ref{eq:u-exp-kernel}) can be applied.

Joint exponential measures $P_\beta$ are mutually absolutely continuous for all $\beta\geq0$.  Furthermore, by Corollary~\ref{cr:support} about the support of utility functions $x(a,b)$ and due to normalization of probability measures, condition $P_\beta(A_i\times B_j)=0$ implies $x(a,b)$ is constant on $A_i\times B_j$, and one may extend this to the case $x(a,b)=-\infty$.  As is well known, exponential distributions approximate the
Dirac $\delta$-function for $\beta\rightarrow\infty$.  The corresponding joint probability measures define deterministic transition kernels $\delta_{f(b)}(a)$, where function $f$ is such that $x(f(b),b)=\sup_{a\in A} x(a,b)$, and one may include the case $\sup x(a,b)=\infty$.

\subsection{Qualitative example}
\label{sec:cauchy}

Strict inequalities of Theorem~\ref{th:nofunctions} present an
interesting opportunity for constructing an example such that $\langle
x,p_f\rangle=-\infty$ or $F(p_f)=\infty$ for any deterministic
transition kernel satisfying a proper information constraint
$F(p)\leq\lambda<\overline\lambda$ or a non-trivial expected utility
constraint $\bE_p\{x\}=\langle
x,p\rangle\geq\upsilon>\overline\upsilon_0$.  If solutions $p_\beta$
to the corresponding variational problems exist, then inequalities $\langle
x,p_\beta\rangle>-\infty$ or $F(p_\beta)<\infty$ suggest that a
non-deterministic transition kernel satisfying the same constraints
may have a finite expected utility and information.  Such an example
would provide qualitative rather than quantitative illustration.  Let
us consider one prototypical example.

Let $a\in A$ and $b\in B$ be real variables, and let us consider the
problem of information transmission between $A$ and $B$ that is
optimal with respect to a measurable utility function $x:A\times
B\rightarrow\bR$.  If $b\in(\bR,\mathcal{B},P)$ is a random variable
with known distribution, then the expected utility $\bE_p\{x\}$ is:
\[
\bE_p\{x\}=\int_A\int_B x(a,b)\,dP(a,b)
=\int_BdP(b)\int_A x(a,b)\,dP(a\mid b)
=\int_B\bE_p\{x\mid b\}\,dP(b)
\]
Here $\bE_p\{x\mid b\}$ denotes the conditional expected utility, and
it is maximized by choosing the optimal conditional probability
measure $dP(a\mid b)$.  The maximum of information is communicated by
an injective function $a=f(b)$, defining a deterministic transition
kernel.  The optimal function is such that $x(f(b),b)=\sup_{a\in A}
x(a,b)$.  On the other hand, if no information can be communicated,
then $dP(a\mid b)=dP(a)$.  A deterministic kernel communicating no
information is defined by a constant function.  Note, however, that
one can still choose an optimal constant function $\bar a_1=f(b)$.
Indeed, if $x(a,b)$ is differentiable and concave in $a$, then $\bar
a_1$ is a solution to the equation $\nabla_a\int_B x(a,b)\,dP(b)=0$.  In
particular, if $x(a,b)=-\frac12(a-b)^2$, then $\nabla_a\int_B
x(a,b)\,dP(b)=\int_B (b-a)\,dP(b)$, and $\bar a_1=\int_B
b\,dP(b)=\bE_p\{b\}$, which is the well-known classical method
minimizing mean-squared deviation.  Thus, for constant $f(b)=a_1$
\[
\bE_{p_f}\{x\}=-\frac12\int_B(a_1-b)^2\,dP(b)
\leq-\frac12\int_B(\bE_p\{b\}-b)^2\,dP(b)
=-\frac12\mathrm{Var}\{b\}
\]
The value on the right depends on the distribution $P(B)$, and there
are many examples of distributions with unbounded variance, such as
$dP(b)=[\pi(b^2+1)]^{-1}\,db$ (the Cauchy distribution).  Indeed, the
integral $\int_B (a-b)^2(b^2+1)^{-1}\,db$ does not converge on
$B=(-\infty,\infty)$.

Let us assume now that some limited information can be communicated so
that $dP(a\mid b)\neq dP(a)$ (and hence $dP(b\mid a)\neq dP(b)$).  For
example, this can be the information associated with $b$ belonging to
some subset of $B$, such as $b>0$ or $b\leq0$.  In each case, one can
choose different optimal elements $\bar a_1$ and $\bar a_2$.  A more
`precise' information would correspond to a larger number of subsets
$B_i\subset B$ and optimal elements $\bar a_i$, such that
\[
\bE_{p_f}\{x\}\leq-\frac12\sum_{i=1}^n\int_{B_i} (\bar a_i-b)^2\,dP(b)
\]
Observe that the value above still depends on $P(B)$, and because for
any finite partition of the real line there are some unbounded
intervals, one can take $P(B)$ giving a negatively infinite value on
the right.  For example, if $P(B)$ is the Cauchy distribution, then
the integral $\int(a-b)^2(b^2+1)\,db$ does not converge on the
intervals $B_1=(-\infty,0]$ or $B_2=[0,\infty)$.  Thus, $b$ can be
  distributed in such a way that the expected value of utility
  $x(a,b)=-\frac12(a-b)^2$ cannot be larger than $-\infty$ for any
  deterministic $p_f$ with finite image $|f(B)|$.  The expected
  utility can have finite values only if $f$ has an infinite image.
  By the argument of Proposition~\ref{pr:maximising-input}, however,
  this means that the function can communicate an infinite amount of
  information.  Let us show now that there exist non-deterministic
  transition kernels for this problem achieving finite expected
  utility and communicating finite amount of information.

Indeed, consider an exponential kernel from
Section~\ref{sec:exponential-kernels}, optimal for constraints on
Shannon mutual information.  Because the utility function
$x(a,b)=-\frac12(a-b)^2$ is translation invariant $x(a+c,b+c)=x(a,b)$,
we can use the simplified expressions from
Section~\ref{sec:exponential-kernels}.  In particular,
$\Psi_0(\beta)=\ln\sqrt{2\pi\beta^{-1}}$, and the exponential kernel
is Gaussian
\[
dP_\beta(a\mid b)=\frac1{\sqrt{2\pi\beta^{-1}}}\,e^{-\beta\frac12(a-b)^2}\,da
\]
Conditional expectation $\bE_{p_\beta}\{x\mid b\}$ is constant for all
$b\in B$:
\[
\bE_{p_\beta}\{x\mid b\}
=-\frac12\frac1{\sqrt{2\pi\beta^{-1}}}\int^{\infty}_{-\infty}(a-b)^2\,e^{-\beta\frac12(a-b)^2}\,da
=-\frac12\frac{\sqrt{2\pi\beta^{-3}}}{\sqrt{2\pi\beta^{-1}}}=-\frac12\beta^{-1}
\]
and therefore
\[
\bE_{p_\beta}\{x\}=\int_B\bE_{p_\beta}\{x\mid b\}\,dP(b)=-\frac12\,\beta^{-1}
\]
The expression above can also be easily obtained from
equation~(\ref{eq:u-exp-kernel}) as the derivative of
$\Psi_0(\beta)=\ln\sqrt{2\pi\beta^{-1}}$.  The optimal value
$\beta^{-1}\geq0$ depends on the amount $\lambda$ of mutual
information, and it can be computed using
equation~(\ref{eq:i-exp-kernel}) by inverting $\lambda=I_S\{a,b\}$:
\[
\beta=2\pi e^{1-2[H\{b\}-\lambda]}
\]
The value $\beta$ depends on the difference $H\{b\}-\lambda$, which
equals to the conditional differential entropy $H\{b\mid a\}$, because
$I_S\{a,b\}=H\{b\}-H\{b\mid a\}=\lambda$.  Therefore, if $H\{b\mid
a\}$ is finite, then $\beta>0$, and $\bE_{p_\beta}\{x\}$ is finite for
all $\lambda>0$.

Other examples can be constructed using the same principles.  For
instance, if $A=B=\bN$, and the utility function $x(a,b)$ is a
polynomial of degree $m\geq1$, then one can distribute $b\in B$
according to $P(b)=[b^{m+1}\zeta(m+1)]^{-1}$, where
$\zeta(k)=\sum_{b\in\bN}b^{-k}$ is the Riemann zeta function.  In this
case, the expected utility is negatively infinite for any
deterministic kernel $\delta_{f(b)}(a)$, if $f$ has finite image
satisfying a finite information constraint.  The optimal transition
kernels satisfying both finite expected utility and finite information
constraints in such problems are non-deterministic.  These examples
demonstrate that deterministic and non-deterministic transition
kernels are qualitatively different, because their expected utilities
can be separated by infinity.

\subsection{Application: Deterministic and non-deterministic algorithms}
\label{sec:algorithms}

Because Markov transition kernels give a non-deterministic generalization of functions, they can be used to model various input-output or information processing systems.  Computational machines and algorithms are examples of such systems, and we now discuss how they can be represented by transition kernels and the corresponding variational problems.  Results of this work may have interesting applications to the study of algorithms and computation.

An algorithm $\Gamma$ is defined as a system of computations transforming input words $w_0$ in some finite alphabet into output (e.g. final) words $w_t$ (e.g. \cite{Markov-Nagornyi88}).  Each word in the domain of definition of $\Gamma$ can be considered as initial word $w_0$.  In a deterministic algorithm, the computation process is performed by a sequence of transformations $\gamma(w_t)=w_{t+1}$ of words, where $\gamma$ is called the {\em direct processing} operator \cite{Kolmogorov-Uspensky58} or a transition function.  In a non-deterministic algorithm, these transitions are randomized according to some local probabilities.  The computational process may terminate reaching a final word (answer), terminate without reaching a final word (error) or continue the computations indefinitely.  In addition, when computation terminates with a non-final word, one may distinguish between errors of the first and second kinds (i.e. false positives and false negatives).  Algorithms may be restricted to run in polynomial time of the size of input words or produce only certain types of errors (i.e. one-sided errors).

The computational cost of $\Gamma(w_0)$ can be associated with resources or complexity of computations, such as the length of the output sequence $(w_1,\dots,w_t)$, if $w_t$ is final:
\[
l(\Gamma(w_0),w_0):=\left\{\begin{array}{ll}
t& \mbox{ if $\Gamma(w_0)=(w_1,\dots,w_t)$ and $w_t$ is a final word}\\
\infty & \mbox{ otherwise}
\end{array}\right.
\]
A Boolean loss function can be defined by $\delta_\infty(l(\Gamma(w_0),w_0))$, where $\delta_\infty(\cdot)$ indicates an error (i.e. one, if the algorithm does not terminate or terminates with a non-final word).  A utility of computation can be defined by any function proportional to negative loss, such as Boolean utility $x(\Gamma(w_0),w_0)=1-\delta_\infty(l(\Gamma(w_0),w_0))$.  Maximization of expectation $\bE_p\{x\}$ for Boolean utility is maximization of the probability that computation terminates with a final word.

Both deterministic and non-deterministic algorithms compute a function from the set of input words $w_0$, for which the computation terminates with an answer, onto the set of final words $w_t$.  The main difference is that a non-deterministic algorithm can compute the pair $(w_0,w_t)$ in different ways and with different running times, so that the cost or utility of a non-deterministic computation is a random variable.  We can represent algorithms by Markov transition kernels as follows.

Let $B$ be the set of all input words $w_0$, and let $A$ be the set of all, possibly infinite, output word sequences $\{w_t\}$.  A deterministic algorithm corresponds to a deterministic Markov transition kernel $\delta_{\Gamma(b)}(a)$, so that each input word is mapped to a particular output word sequence: $B\ni w_0\mapsto\Gamma(w_0)=(w_1,\dots,w_t,\dots)\in A$.  A non-deterministic algorithm assigns non-zero probabilities $P_\Gamma(a\mid b)$ to different output sequences.  We say that two algorithms are equivalent, if they correspond to identical Markov transition kernels.  Points in the set $\cP(A\times B)$, which is a Choquet simplex, correspond to equivalence classes of all deterministic and non-deterministic algorithms, defined on $B$, together with all distributions $P(B)$ of input words.  This formalism allows us to consider optimization of algorithms in the context of variational problems~(\ref{eq:max-i}), (\ref{eq:min-i}) and their generalizations.

Indeed, optimization of a class of algorithms subject to constraint $\bE_p\{l\}\leq\upsilon$ on the expected loss or a constraint $\bE_p\{x\}\geq\upsilon$ on the expected utility has been considered in complexity theory (e.g. see \cite{Goldreich08}).  For example, the complexity class of bounded error probabilistic polynomial time machines (BPP) is defined as a class of problems solved by non-deterministic algorithms with constraints on the expected error (i.e. $\bE_p\{x\}\geq\upsilon>1/2$, where $x$ is Boolean utility).  Information constraints have also been considered in complexity theory, such as constraints on communication capacity (communication complexity) or in the class of probabilistically checkable proofs (PCP), which is defined as a non-deterministic algorithm with constraints on randomness and a number of queries to an oracle (i.e. a constraint on information amount about the proof).  Problems of optimization of algorithms can be considered as a search for the corresponding class of optimal Markov transition kernels (i.e. variational problems of the third kind in information theory).  The optimal value functions~(\ref{eq:support})--(\ref{eq:inverse-anti-support}) put the expected utility constraint $\bE_p\{x\}\geq\upsilon$ in duality with a constraint $F(p)\leq\lambda$ on an information resource.  Thus, the study of performance and computational complexity of the algorithms is related to the study of their information constraints.

\section{Discussion}
\label{sec:discussion}

We have studied families of optimal measures using a generalization of the
classical variational problems of information theory \cite{Shannon48,Stratonovich65} and statistical physics \cite{Jaynes57}.  In fact, standard formulae of these theories relating Gibbs measures, free energy, entropy and channel capacity can be recovered simply by defining information constraints using the Kullback-Leibler divergence.  The main motivation for the generalization was understanding the mutual absolute continuity of measures within optimal families, and it was established that such families exist if an abstract information resource has a strictly convex dual, which is a geometric rather than algebraic property of information.  We have discussed also that strict convexity of the dual functional is related to separability of different variational problems, which is useful in the context of optimization.  Our method does not depend on commutativity of the algebra of random variables or observables, and for this reason the result holds both for commutative (classical) and non-commutative (quantum) measures.

Mutual absolute continuity of optimal probability measures allowed us
to show that deterministic transition kernels are strictly
sub-optimal.  This result is important not only for applications of
optimization theory, but also for some theoretical questions in
studies of algorithms and computational complexity, where much of the
effort is devoted to the question whether non-deterministic procedures
have any qualitative advantage over deterministic.  Our results
suggest that in a broad class of optimization problems with
constraints on information optimal deterministic kernels do not exist.
Moreover, an example has been constructed to show that the difference
between expected utilities of deterministic and non-deterministic
kernels can be infinite for all proper constraints on an information
resource.

These results about strict sub-optimality of deterministic kernels do
not contradict the established understanding in the classical theory
of statistical decisions that asymptotically randomized policies
cannot be better than deterministic (e.g. see
\cite{Stratonovich75:_inf} or more recently \cite{Kozen-Ruozzi09}).
Indeed, these asymptotic results are concerned with obtaining all,
possibly infinite amount of information, in which case there are
deterministic optimal kernels.  Our results, on the other hand, are
about optimality subject to constraints making such asymptotic
solutions unfeasible.  Note also that a simple randomization of a
function's output can only decrease (loose) the amount of information
it communicates.  However, we have compared deterministic and
non-deterministic kernels that can communicate the same amount of
information.  The possibility to separate deterministic and
non-deterministic transitions qualitatively (i.e. by infinity) is
particularly interesting, because it confirms a common intuition in
applied optimization about numerous problems, in which
non-deterministic algorithms outperform all known deterministic
methods.

\paragraph{Acknowledgements}
I would like to express my gratitude to Paul Blampied, Vladimir Goncharov, Pando Georgiev, Satoshi Iriyama and Serguei Novak for valuable discussions of the early drafts of this paper.  Special thanks go my father, Viacheslav Belavkin, for clarifying some algebraic and non-commutative issues, and to my mother for her support during these discussions.  I am also indebted to my girlfriend Oliya for her love and inspiration.  This work was supported by the United Kingdom Engineering and Physical Sciences Research Council (EPSRC) grant EP/H031936/1.

\bibliographystyle{spmpsci}      

\bibliography{rvb,nn,other,newbib,ica}

\end{document}

%% file: simplex-kl.tex
\begingroup
  \makeatletter
  \providecommand\color[2][]{%
    \GenericError{(gnuplot) \space\space\space\@spaces}{%
      Package color not loaded in conjunction with
      terminal option `colourtext'%
    }{See the gnuplot documentation for explanation.%
    }{Either use 'blacktext' in gnuplot or load the package
      color.sty in LaTeX.}%
    \renewcommand\color[2][]{}%
  }%
  \providecommand\includegraphics[2][]{%
    \GenericError{(gnuplot) \space\space\space\@spaces}{%
      Package graphicx or graphics not loaded%
    }{See the gnuplot documentation for explanation.%
    }{The gnuplot epslatex terminal needs graphicx.sty or graphics.sty.}%
    \renewcommand\includegraphics[2][]{}%
  }%
  \providecommand\rotatebox[2]{#2}%
  \@ifundefined{ifGPcolor}{%
    \newif\ifGPcolor
    \GPcolorfalse
  }{}%
  \@ifundefined{ifGPblacktext}{%
    \newif\ifGPblacktext
    \GPblacktexttrue
  }{}%
  \let\gplgaddtomacro\g@addto@macro
  \gdef\gplbacktext{}%
  \gdef\gplfronttext{}%
  \makeatother
  \ifGPblacktext
    \def\colorrgb#1{}%
    \def\colorgray#1{}%
  \else
    \ifGPcolor
      \def\colorrgb#1{\color[rgb]{#1}}%
      \def\colorgray#1{\color[gray]{#1}}%
      \expandafter\def\csname LTw\endcsname{\color{white}}%
      \expandafter\def\csname LTb\endcsname{\color{black}}%
      \expandafter\def\csname LTa\endcsname{\color{black}}%
      \expandafter\def\csname LT0\endcsname{\color[rgb]{1,0,0}}%
      \expandafter\def\csname LT1\endcsname{\color[rgb]{0,1,0}}%
      \expandafter\def\csname LT2\endcsname{\color[rgb]{0,0,1}}%
      \expandafter\def\csname LT3\endcsname{\color[rgb]{1,0,1}}%
      \expandafter\def\csname LT4\endcsname{\color[rgb]{0,1,1}}%
      \expandafter\def\csname LT5\endcsname{\color[rgb]{1,1,0}}%
      \expandafter\def\csname LT6\endcsname{\color[rgb]{0,0,0}}%
      \expandafter\def\csname LT7\endcsname{\color[rgb]{1,0.3,0}}%
      \expandafter\def\csname LT8\endcsname{\color[rgb]{0.5,0.5,0.5}}%
    \else
      \def\colorrgb#1{\color{black}}%
      \def\colorgray#1{\color[gray]{#1}}%
      \expandafter\def\csname LTw\endcsname{\color{white}}%
      \expandafter\def\csname LTb\endcsname{\color{black}}%
      \expandafter\def\csname LTa\endcsname{\color{black}}%
      \expandafter\def\csname LT0\endcsname{\color{black}}%
      \expandafter\def\csname LT1\endcsname{\color{black}}%
      \expandafter\def\csname LT2\endcsname{\color{black}}%
      \expandafter\def\csname LT3\endcsname{\color{black}}%
      \expandafter\def\csname LT4\endcsname{\color{black}}%
      \expandafter\def\csname LT5\endcsname{\color{black}}%
      \expandafter\def\csname LT6\endcsname{\color{black}}%
      \expandafter\def\csname LT7\endcsname{\color{black}}%
      \expandafter\def\csname LT8\endcsname{\color{black}}%
    \fi
  \fi
  \setlength{\unitlength}{0.0500bp}%
  \begin{picture}(4320.00,4284.00)%
    \gplgaddtomacro\gplbacktext{%
      \csname LTb\endcsname%
      \put(2214,1310){\makebox(0,0)[l]{\strut{}$q$}}%
      \put(1839,2078){\makebox(0,0)[r]{\strut{}$p_\beta$}}%
      \put(1900,2522){\makebox(0,0)[l]{\strut{}$\bE_p\{x\}\geq\upsilon$}}%
      \put(1641,742){\makebox(0,0)[l]{\strut{}$\bE_p\{\ln(p/q)\}\leq\lambda$}}%
    }%
    \gplgaddtomacro\gplfronttext{%
    }%
    \gplbacktext
    \put(0,0){\includegraphics{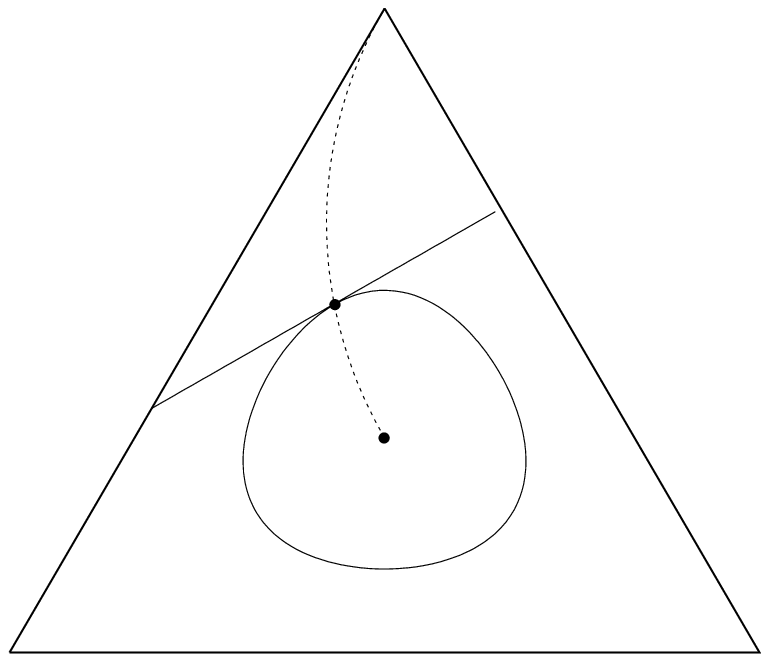}}%
    \gplfronttext
  \end{picture}%
\endgroup

%% file: ui-branches.tex
\begingroup
  \makeatletter
  \providecommand\color[2][]{%
    \GenericError{(gnuplot) \space\space\space\@spaces}{%
      Package color not loaded in conjunction with
      terminal option `colourtext'%
    }{See the gnuplot documentation for explanation.%
    }{Either use 'blacktext' in gnuplot or load the package
      color.sty in LaTeX.}%
    \renewcommand\color[2][]{}%
  }%
  \providecommand\includegraphics[2][]{%
    \GenericError{(gnuplot) \space\space\space\@spaces}{%
      Package graphicx or graphics not loaded%
    }{See the gnuplot documentation for explanation.%
    }{The gnuplot epslatex terminal needs graphicx.sty or graphics.sty.}%
    \renewcommand\includegraphics[2][]{}%
  }%
  \providecommand\rotatebox[2]{#2}%
  \@ifundefined{ifGPcolor}{%
    \newif\ifGPcolor
    \GPcolorfalse
  }{}%
  \@ifundefined{ifGPblacktext}{%
    \newif\ifGPblacktext
    \GPblacktexttrue
  }{}%
  \let\gplgaddtomacro\g@addto@macro
  \gdef\gplbacktext{}%
  \gdef\gplfronttext{}%
  \makeatother
  \ifGPblacktext
    \def\colorrgb#1{}%
    \def\colorgray#1{}%
  \else
    \ifGPcolor
      \def\colorrgb#1{\color[rgb]{#1}}%
      \def\colorgray#1{\color[gray]{#1}}%
      \expandafter\def\csname LTw\endcsname{\color{white}}%
      \expandafter\def\csname LTb\endcsname{\color{black}}%
      \expandafter\def\csname LTa\endcsname{\color{black}}%
      \expandafter\def\csname LT0\endcsname{\color[rgb]{1,0,0}}%
      \expandafter\def\csname LT1\endcsname{\color[rgb]{0,1,0}}%
      \expandafter\def\csname LT2\endcsname{\color[rgb]{0,0,1}}%
      \expandafter\def\csname LT3\endcsname{\color[rgb]{1,0,1}}%
      \expandafter\def\csname LT4\endcsname{\color[rgb]{0,1,1}}%
      \expandafter\def\csname LT5\endcsname{\color[rgb]{1,1,0}}%
      \expandafter\def\csname LT6\endcsname{\color[rgb]{0,0,0}}%
      \expandafter\def\csname LT7\endcsname{\color[rgb]{1,0.3,0}}%
      \expandafter\def\csname LT8\endcsname{\color[rgb]{0.5,0.5,0.5}}%
    \else
      \def\colorrgb#1{\color{black}}%
      \def\colorgray#1{\color[gray]{#1}}%
      \expandafter\def\csname LTw\endcsname{\color{white}}%
      \expandafter\def\csname LTb\endcsname{\color{black}}%
      \expandafter\def\csname LTa\endcsname{\color{black}}%
      \expandafter\def\csname LT0\endcsname{\color{black}}%
      \expandafter\def\csname LT1\endcsname{\color{black}}%
      \expandafter\def\csname LT2\endcsname{\color{black}}%
      \expandafter\def\csname LT3\endcsname{\color{black}}%
      \expandafter\def\csname LT4\endcsname{\color{black}}%
      \expandafter\def\csname LT5\endcsname{\color{black}}%
      \expandafter\def\csname LT6\endcsname{\color{black}}%
      \expandafter\def\csname LT7\endcsname{\color{black}}%
      \expandafter\def\csname LT8\endcsname{\color{black}}%
    \fi
  \fi
  \setlength{\unitlength}{0.0500bp}%
  \begin{picture}(6480.00,4536.00)%
    \gplgaddtomacro\gplbacktext{%
      \csname LTb\endcsname%
      \put(792,815){\makebox(0,0)[r]{\strut{}$\underline\upsilon$}}%
      \put(792,2123){\makebox(0,0)[r]{\strut{}$\underline\upsilon_0$}}%
      \put(792,2407){\makebox(0,0)[r]{\strut{}$\overline\upsilon_0$}}%
      \put(792,3952){\makebox(0,0)[r]{\strut{}$\overline\upsilon$}}%
      \put(1577,374){\makebox(0,0){\strut{}$\lambda_0$}}%
      \put(4190,374){\makebox(0,0){\strut{}$\underline\lambda$}}%
      \put(5496,374){\makebox(0,0){\strut{}$\overline\lambda$}}%
      \put(286,2564){\rotatebox{-270}{\makebox(0,0){\strut{}Optimal values, $\upsilon=\langle x,y\rangle$}}}%
      \put(3536,154){\makebox(0,0){\strut{}Constraint values, $\lambda\geq F(y)$}}%
      \put(2100,2959){\makebox(0,0)[l]{\strut{}$\overline x(\lambda):=\sup\{\langle x,y\rangle:F(y)\leq\lambda\}$}}%
      \put(2100,1650){\makebox(0,0)[l]{\strut{}$\underline x(\lambda):=\inf\{\langle x,y\rangle:F(y)\leq\lambda\}$}}%
    }%
    \gplgaddtomacro\gplfronttext{%
    }%
    \gplbacktext
    \put(0,0){\includegraphics{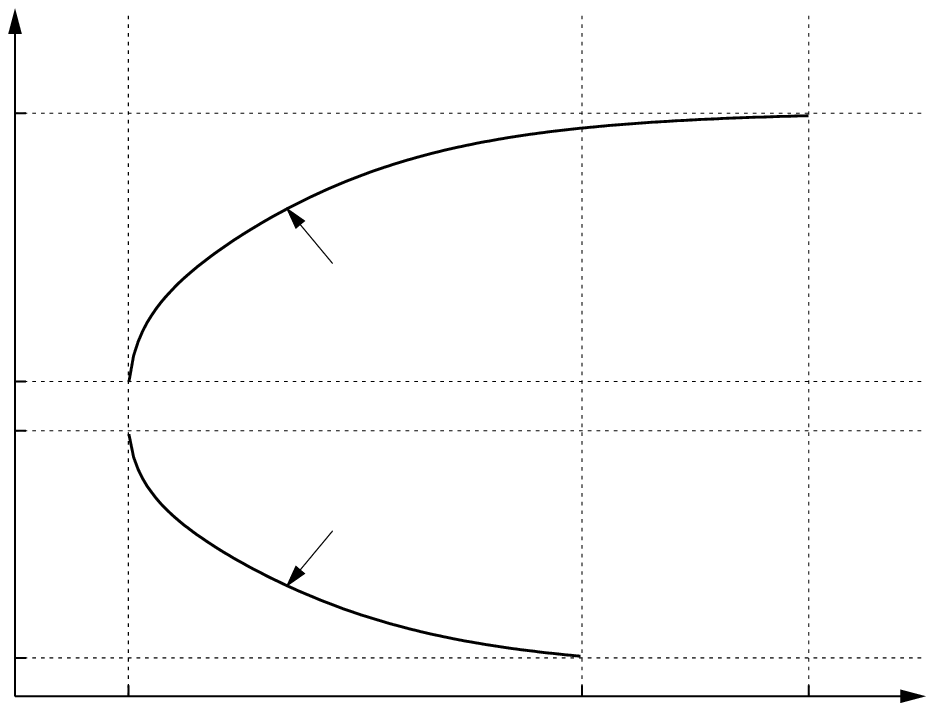}}%
    \gplfronttext
  \end{picture}%
\endgroup

%% file: simplex-l1.tex
\begingroup
  \makeatletter
  \providecommand\color[2][]{%
    \GenericError{(gnuplot) \space\space\space\@spaces}{%
      Package color not loaded in conjunction with
      terminal option `colourtext'%
    }{See the gnuplot documentation for explanation.%
    }{Either use 'blacktext' in gnuplot or load the package
      color.sty in LaTeX.}%
    \renewcommand\color[2][]{}%
  }%
  \providecommand\includegraphics[2][]{%
    \GenericError{(gnuplot) \space\space\space\@spaces}{%
      Package graphicx or graphics not loaded%
    }{See the gnuplot documentation for explanation.%
    }{The gnuplot epslatex terminal needs graphicx.sty or graphics.sty.}%
    \renewcommand\includegraphics[2][]{}%
  }%
  \providecommand\rotatebox[2]{#2}%
  \@ifundefined{ifGPcolor}{%
    \newif\ifGPcolor
    \GPcolorfalse
  }{}%
  \@ifundefined{ifGPblacktext}{%
    \newif\ifGPblacktext
    \GPblacktexttrue
  }{}%
  \let\gplgaddtomacro\g@addto@macro
  \gdef\gplbacktext{}%
  \gdef\gplfronttext{}%
  \makeatother
  \ifGPblacktext
    \def\colorrgb#1{}%
    \def\colorgray#1{}%
  \else
    \ifGPcolor
      \def\colorrgb#1{\color[rgb]{#1}}%
      \def\colorgray#1{\color[gray]{#1}}%
      \expandafter\def\csname LTw\endcsname{\color{white}}%
      \expandafter\def\csname LTb\endcsname{\color{black}}%
      \expandafter\def\csname LTa\endcsname{\color{black}}%
      \expandafter\def\csname LT0\endcsname{\color[rgb]{1,0,0}}%
      \expandafter\def\csname LT1\endcsname{\color[rgb]{0,1,0}}%
      \expandafter\def\csname LT2\endcsname{\color[rgb]{0,0,1}}%
      \expandafter\def\csname LT3\endcsname{\color[rgb]{1,0,1}}%
      \expandafter\def\csname LT4\endcsname{\color[rgb]{0,1,1}}%
      \expandafter\def\csname LT5\endcsname{\color[rgb]{1,1,0}}%
      \expandafter\def\csname LT6\endcsname{\color[rgb]{0,0,0}}%
      \expandafter\def\csname LT7\endcsname{\color[rgb]{1,0.3,0}}%
      \expandafter\def\csname LT8\endcsname{\color[rgb]{0.5,0.5,0.5}}%
    \else
      \def\colorrgb#1{\color{black}}%
      \def\colorgray#1{\color[gray]{#1}}%
      \expandafter\def\csname LTw\endcsname{\color{white}}%
      \expandafter\def\csname LTb\endcsname{\color{black}}%
      \expandafter\def\csname LTa\endcsname{\color{black}}%
      \expandafter\def\csname LT0\endcsname{\color{black}}%
      \expandafter\def\csname LT1\endcsname{\color{black}}%
      \expandafter\def\csname LT2\endcsname{\color{black}}%
      \expandafter\def\csname LT3\endcsname{\color{black}}%
      \expandafter\def\csname LT4\endcsname{\color{black}}%
      \expandafter\def\csname LT5\endcsname{\color{black}}%
      \expandafter\def\csname LT6\endcsname{\color{black}}%
      \expandafter\def\csname LT7\endcsname{\color{black}}%
      \expandafter\def\csname LT8\endcsname{\color{black}}%
    \fi
  \fi
  \setlength{\unitlength}{0.0500bp}%
  \begin{picture}(4320.00,4284.00)%
    \gplgaddtomacro\gplbacktext{%
      \csname LTb\endcsname%
      \put(2214,1310){\makebox(0,0)[l]{\strut{}$q$}}%
      \put(1688,2048){\makebox(0,0)[r]{\strut{}$p_\beta$}}%
      \put(1900,2522){\makebox(0,0)[l]{\strut{}$\bE_p\{x\}\geq\upsilon$}}%
      \put(1728,742){\makebox(0,0)[l]{\strut{}$\|p-q\|_1\leq\lambda$}}%
      \put(2267,2114){\makebox(0,0)[l]{\strut{}$\bE_p\{w\}\geq\upsilon$}}%
    }%
    \gplgaddtomacro\gplfronttext{%
    }%
    \gplbacktext
    \put(0,0){\includegraphics{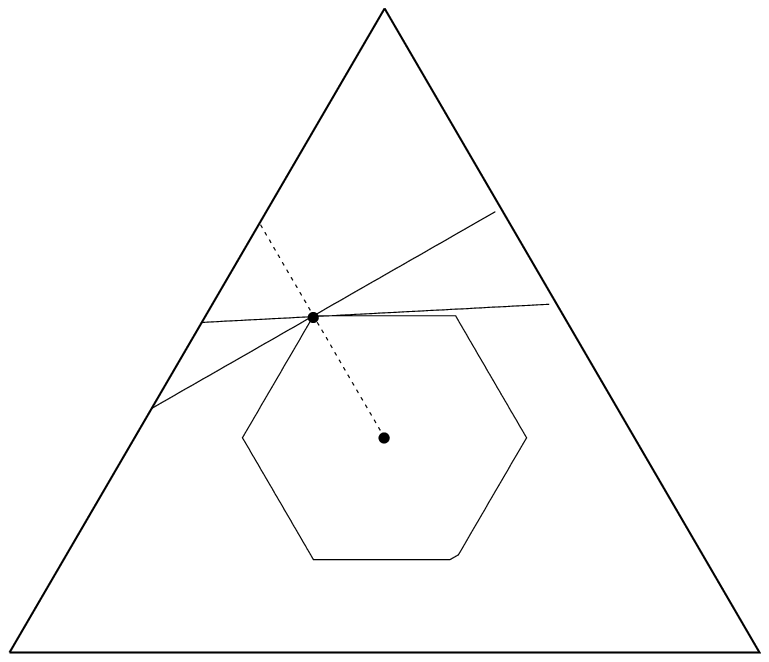}}%
    \gplfronttext
  \end{picture}%
\endgroup